\documentclass[11pt]{amsart}
\usepackage{amscd,amsmath,amssymb,amsfonts,verbatim}
\usepackage[cmtip, all]{xy}
\usepackage{comment}

\usepackage{xcolor}
\newcommand{\blue}[1]{\textcolor{blue}{#1}}

\newtheorem{thm}{Theorem}[section]
\newtheorem{prop}[thm]{Proposition}
\newtheorem{lem}[thm]{Lemma}
\newtheorem{cor}[thm]{Corollary}

\theoremstyle{definition}

\newtheorem{defn}[thm]{Definition}
\theoremstyle{remark}
\newtheorem{remk}[thm]{Remark}
\newtheorem{remks}[thm]{Remarks}

\newtheorem{exm}[thm]{Example}
\newtheorem{exms}[thm]{Examples}
\newtheorem{notat}[thm]{Notation}
\numberwithin{equation}{section}

{\hfill$\square$\end{defn}}
{\hfill$\square$\end{remk}}
{\hfill$\square$\end{remks}}
{\hfill$\square$\end{exm}}
{\hfill$\square$\end{exms}}
{\hfill$\square$\end{notat}}

\newcommand{\thmref}{Theorem~\ref}
\newcommand{\propref}{Proposition~\ref}
\newcommand{\corref}{Corollary~\ref}

\newcommand{\lemref}{Lemma~\ref}

\newcommand{\sA}{{\mathcal A}}

\newcommand{\sC}{{\mathcal C}}
\newcommand{\sD}{{\mathcal D}}

\newcommand{\sF}{{\mathcal F}}

\newcommand{\sH}{{\mathcal H}}
\newcommand{\sI}{{\mathcal I}}
\newcommand{\sJ}{{\mathcal J}}
\newcommand{\sK}{{\mathcal K}}

\newcommand{\sO}{{\mathcal O}}

\newcommand{\sHH}{{\mathcal H}{\mathcal H}}
\newcommand{\sHC}{{\mathcal H}{\mathcal C}}

\newcommand{\tHC}{\widetilde{{\mathcal H}{\mathcal C}}}

\newcommand{\A}{{\mathbb A}}

\newcommand{\N}{{\mathbb N}}
\renewcommand{\P}{{\mathbb P}}
\newcommand{\Q}{{\mathbb Q}}

\newcommand{\Z}{{\mathbb Z}}

\newcommand{\fm}{{\mathfrak m}}
\newcommand{\fp}{{\mathfrak p}}

\newcommand{\surj}{\twoheadrightarrow}
\newcommand{\inj}{\hookrightarrow}

\newcommand{\Hom}{{\rm Hom}}

\newcommand{\Spec}{{\rm Spec \,}}

\newcommand{\Ker}{{\rm Ker}}

\newcommand{\id}{{\operatorname{id}}}

\newcommand{\eft}{{\operatorname{\mathbf{Eftalg}}}} 
\newcommand{\rft}{{\operatorname{\mathbf{Reftalg}}}} 
\newcommand{\Sch}{{\operatorname{\mathbf{Sch}}}}

\newcommand{\Sm}{{\mathbf{Sm}}}

\newcommand{\GL}{{\operatorname{\rm GL}}}
\newcommand{\SL}{{\operatorname{\rm SL}}}

\newcommand{\ds}{{/\kern-3pt/}}

\newcommand{\Proj}{{\operatorname{Proj}}}

\newcommand{\ov}{\overline}

\newcommand{\tuborg}{\left\{\begin{array}{ll}}
\newcommand{\sluttuborg}{\end{array}\right.}

\newcommand{\zar}{{\rm zar}}

\newcommand{\wt}{\widetilde}

\newcounter{elno}

\newcounter{elno-abc}   

\newcounter{elno-abc-prime}

\begin{document}
\title{$K$-theory of monoid algebras and a question of Gubeladze}
\author{Amalendu Krishna, Husney Parvez Sarwar}
\address{School of Mathematics, Tata Institute of Fundamental Research,  
1 Homi Bhabha Road, Colaba, Mumbai, India}
\email{amal@math.tifr.res.in}
\email{mathparvez@gmail.com}




\maketitle

\begin{quote}\emph{Abstract.}  
We show that for any commutative 
Noetherian regular ring $R$ containing $\Q$, the map
$K_1(R) \to K_1(\frac{R[x_1, \cdots , x_4]}{(x_1x_2 - x_3x_4)})$ is
an isomorphism. This answers a question of Gubeladze. We also
compute the higher $K$-theory of this monoid algebra. In particular,
we show that the above isomorphism does not extend to all higher $K$-groups.
We give applications to a question of Lindel on the Serre
dimension of monoid algebras.
\end{quote}
\setcounter{tocdepth}{1}
\tableofcontents

\section{Introduction}\label{sec:Intro}
The algebraic $K$-theory is well known to be a very powerful invariant to study
various geometric and cohomological properties of algebraic varieties.
Quillen \cite{Quillen} showed that algebraic $K$-theory of 
smooth varieties satisfies homotopy invariance.
This implies in particular that the algebraic $K$-theory 
of commutative Noetherian regular rings remains invariant under any
polynomial extension. 

The monoid algebras associated to finitely generated monoids
are natural generalizations of polynomial
algebras over commutative rings. These algebras form a very important
class of toy models
to test various properties of algebraic $K$-theory which are known to be
true for polynomial algebras. Questions like Serre's problem on 
projective modules, homotopy invariance of $K$-theory, and
$K$-regularity have been extensively
studied for monoid algebras by various authors, see for instance,
\cite{Anderson}, \cite{Gub-1}, \cite{Gub-2} and \cite{CHW-1}.

In a series of several papers, Gubeladze \cite{Gub-0}, \cite{Gub-3}
showed that for a regular commutative 
Noetherian ring $R$ and a monoid $M$, one has
$K_i(R[M]) = 0$ for $i < 0$ and $K_0(R) \xrightarrow{\simeq} K_0(R[M])$,
if $M$ is semi-normal. In particular, the analogue of
homotopy invariance extends to $K$-theory of monoid algebras over regular
rings in degree up to zero. However, Srinivas \cite{Srinivas}
showed that this property no longer holds in higher degrees, by showing
that $SK_1(k[M]) \neq 0$, where $k$ is any algebraically closed field of 
characteristic different from two and $k[M]$ is the monoid algebra
$k[x^2_1, x_1x_2, x^2_2] \subset k[x_1, x_2]$.

Gubeladze \cite{Gub-1} gave a different and more algebraic proof of 
Srinivas' result with no condition on the ground field $k$. 
He further showed that the above monoid algebra is not $K_1$-regular.  
This motivates the following question (See \cite[Question, pp 170]{Gub-1}). 
Let $k$ be a field and let $M$ be a finitely generated 
normal, cancellative, torsion-free monoid which has no 
non-trivial units and which is not $c$-divisible for any integer $c \neq 1$
(see \S~\ref{sec:Not} for definitions).
Assume that 
\begin{enumerate}
\item
$k[M]$ is $K_1$-regular and
\item
$SK_1(k[M]) = 0$.
\end{enumerate}
Is $M \simeq \N^r$?

Note that $c$-divisibility is very important here in view of 
\cite[Theorem~1.12]{Gub-2}. 
For any commutative ring $R$, let $R[M]$ denote the monoid algebra
\begin{equation}\label{eqn:Monoid-def}
R[M] =  R[x_1x_3, x_1x_4, x_2x_3, x_2x_4] \subset R[x_1, \cdots , x_4].
\end{equation}

From the geometric point of view, it is useful to note that $R[M]$ is 
the homogeneous 
coordinate ring for the Segre embedding of $ \P^1_R \times_R \P^1_R$ 
inside $\P^3_R$. 

When $R$ is a field, Gubeladze (\cite[Remark~1.9 (c)]{Gub-2}) 
asked if the monoid algebra ~\eqref{eqn:Monoid-def} is $K_1$-regular 
{\footnote {In fact, 
Gubeladze asked if this algebra is $K_i$-regular for all $i \ge 1$, but
the resolution of Vorst's conjecture by Corti{\~n}as, Haesemeyer and Weibel
\cite{CHW} shows that such algebras do not exist.}}.
If so, this could give a counter-example to the above question.
We refer to \cite{Gub-2} for Gubeladze's consideration of this monoid algebra 
and for an excellent account of many results and other open questions
on the algebraic $K$-theory of monoid algebras.

The algebraic $K$-theory of many monoid algebras
(including ~\eqref{eqn:Monoid-def}) were extensively studied by  
Corti{\~n}as, Haesemeyer, Walker and Weibel \cite{CHW-1} using the technique of
$cdh$-descent of homotopy invariant $K$-theory.
However, Gubeladze's question remained an open and intractable problem.

\subsection{Gubeladze's question}
The main result of this text is to show that Gubeladze's monoid
algebra is indeed a counter-example to the above question.
In fact, we prove the following more general statement.

\begin{thm}\label{thm:Main-1}
Let $R$ be a commutative Noetherian regular ring containing $\Q$. Then the
inclusion of scalars $R \inj R[M]$
induces an isomorphism 
\[
K_1(R) \xrightarrow{\simeq} K_1(R[M]).
\]
\end{thm}

Using the commutative square
\begin{equation}\label{eqn:K1-reg}
\xymatrix@C1pc{
K_1(R) \ar[r]^-{\simeq} \ar[d]_{\simeq} & K_1(R[M]) \ar[d] \\
K_1(R[X_1, \cdots , X_n]) \ar[r]_-{\simeq} &
K_1(R[M][X_1, \cdots , X_n]),}
\end{equation}
\thmref{thm:Main-1} and the homotopy invariance of $K$-theory of regular rings
together imply the following.

\begin{cor}\label{cor:K1-reg*}
Let $R$ be a commutative Noetherian regular ring containing $\Q$. Then
$R[M]$ is $K_1$-regular.
\end{cor}

\subsection{Lindel's question}
As another application of \thmref{thm:Main-1}, we can give a partial answer to
a question of Lindel \cite{L95} on the Serre dimension of monoid algebras. 
Recall that a commutative Noetherian ring $R$ is said to 
have Serre dimension (denoted by {\sl Serre dim}$(R)$)
at most $t \ge 0$, if every projective $R$-module of
rank higher than $t$ splits off a rank one free direct summand.
Lindel studied the Serre dimension of the 
monoid algebra $R[M] = R[x_1x_3, x_1x_4, x_2x_3, x_2x_4]$ in  \cite{L95}.
He showed that {\sl Serre dim}$(R[M]) \le d + 1$,
where $R$ is any commutative Noetherian ring of Krull dimension $d$.

Motivated by his results, Lindel \cite{L95} asked if one actually has
{\sl Serre dim}$(R[M]) \le d$ in the above case.
Note that this bound on {\sl Serre dim}$(R[M])$
is the best possible for a general Noetherian ring $R$.
We apply \thmref{thm:Main-1} to prove the following result which is
due to Swan (\cite[Theorem~1.1]{Swan}) when $R$ is a Dedekind domain.

\begin{thm}\label{thm:Main-2}
Let $R$ be an one-dimensional Noetherian commutative ring containing $\Q$.
Let $P$ be a projective module over $R[M]$
of rank $r \ge 2$. Then $P \simeq \wedge^r(P) \oplus {R[M]}^{r-1}$.
In particular, $\mbox{\sl Serre dim}(R[M]) \le 1$.
\end{thm}

\subsection{Higher $K$-theory of $k[M]$}
When $k$ is a field which is algebraic over $\Q$, we can improve
\thmref{thm:Main-1} to prove the following more general result.

\begin{thm}\label{thm:Mian-3}
Let $k$ be a field which is algebraic over $\Q$.
Then for any integer $i \notin \{2,3,4\}$, one has
\[
K_i(k) \xrightarrow{\simeq} K_i(k[M]).
\]
\end{thm}

We can apply  \thmref{thm:Mian-3} to show that it can not
be extended to all polynomial algebras over $k$ for $i\geq 5$.
Indeed, if $K_i(k[X_1,\ldots, X_n]) \xrightarrow{\simeq} 
K_i(k[M][X_1,\ldots,X_n])$ for all $n \ge 1$, the homotopy invariance
of $K_*(k)$ and \thmref{thm:Mian-3} together would tell us that
$k[M]$ is $K_5$-regular. In particular, it is $K_4$-regular by
\cite[Corollary~2.1 (ii)]{Vorst79}. But this is not possible in view of 
\cite{CHW}.

\thmref{thm:Main-1} shows that Gubeladze's monoid algebra shares some
important $K$-theoretic properties (see the question before 
~\eqref{eqn:Monoid-def} for other common properties) of free monoid 
algebras even though it is not free. This raises the following question:
can algebraic $K$-theory actually detect this defect?

We answer this question affirmatively
by giving an explicit  expression for all $K$-groups of the
above monoid algebra in \thmref{thm:Main-alg-K}.
Using these computations,  
we show that the isomorphisms of Theorems~\ref{thm:Main-1}
and ~\ref{thm:Mian-3} do not extend to all $K$-groups.
In particular, the algebraic $K$-theory does explain the non-freeness of
Gubeladze's monoid algebra.

\begin{thm}\label{thm:Main-4}
Let $k$ be any field of characteristic zero. Then
the map $K_4(k) \to K_4(k[M])$ is not an
isomorphism. In particular, the map of spectra $K(k) \to 
K(k[M])$ is not a weak equivalence.
\end{thm}


It is known that the reduced algebraic $K$-theory 
(${\rm Ker}(K_*(k[M]) \to K_*(k)$) of the above monoid algebra  
consists of uniquely divisible groups $\wt{K}_*(k[M])$ 
(e.g., see \cite[\S~1]{CHW-1}).
In particular, there is a direct sum decomposition 
$\wt{K}_i(k[M]) \simeq \oplus_{j \ge 0} \wt{K}^{(j)}_i(k[M])$ of eigen pieces
for Adams operation, whenever $i \ge 0$.
For monoid algebras arising from the cones over smooth projective 
schemes, many of these pieces have been computed by
Corti{\~n}as, Haesemeyer, Walker and Weibel
in \cite{CHW-1}. 

However, it is not clear if the technique of $cdh$-descent for the $KH$-theory 
employed in \cite{CHW-1} can be helpful in
computing the $K$-theory of Gubeladze's monoid algebra.
Instead, we use the more recent
pro-$cdh$ descent for the Quillen $K$-theory and work in the category of
pro $K$-groups to make the above breakthrough. In fact,
the referee pointed out that the results of this paper are possibly the very 
first applications of the pro-$cdh$ descent to the study of monoid algebras.

We end the introduction of the results with a note. Some critics may ask if
our methods help in computing the $K$-theory of other monoid algebras.
To this, we remark that even if it may be possible, 
it was not the purpose of this 
work to devise general techniques to study $K$-theory of monoid algebras. 
This is already done in \cite{CHW-1} in the best possible way. Our motivation
was to verify some conjectures and questions of Gubeladze which remained
intractable, even after \cite{CHW-1}. 
And one should note that Gubeladze's question is
about producing counter-examples in $K$-theory of monoid algebras.

\subsection{Outline of the proofs}
As stated above, a new idea which turned out to be very crucial for 
computing the
higher $K$-theory of the monoid algebra of ~\eqref{eqn:Monoid-def}
is the pro-descent theorem of \cite{Krishna-2} and \cite{Morrow}.
This descent theorem tells us that it is often enough to prove many vanishing
theorems for Hochschild and cyclic homology groups only in
the pro-setting in order to compute higher $K$-theory of singular schemes.

In \S~\ref{sec:Review},
we review finitely generated monoids and monoid algebras. 
We also study the geometry of the Zariski spectrum of the monoid
algebra of Theorem~\ref{thm:Main-1} and use this geometry to
compute the cohomology groups of the sheaves of differential forms
on this algebra. The heart of this text are sections ~\ref{sec:R-reg} and
~\ref{sec:Global-sec}, where we prove some vanishing theorems
which are critical for the proofs of the above results.
In \S~\ref{sec:Alg}, we generalize some results of \S~\ref{sec:Global-sec}
in order to compute the higher $K$-theory of our monoid algebra.
In \S~\ref{sec:MT}, we combine the vanishing theorems of 
sections ~\ref{sec:R-reg} and ~\ref{sec:Global-sec} with the 
pro-descent theorem of \cite{Krishna-2} and \cite{Morrow} to
complete the proofs of our main results.

\section{Review of monoid algebras and their geometry}
\label{sec:Review}
In this section, we set up our notations and
review various definitions in the study of monoids
and monoid algebras. We then study the geometry of the monoid
algebra $k[x_1x_3, x_1x_4, x_2x_3, x_2x_4]$, whose algebraic $K$-theory is
the main interest of this text.

In this paper, a ring $R$ will always mean a 
commutative, Noetherian $\Q$-algebra. 
In particular, all fields will be of characteristic zero. Given a ring $R$
as above, we shall let $\eft_R$ denote the category of
essentially of finite type $R$-algebras and let $\rft_R$ denote
the full subcategory of $\eft_R$ consisting of smooth
$R$-algebras. We shall let $\Sch_R$ denote the category of
separated Noetherian schemes over $\Spec(R)$ and let $\Sm_R$ denote the
full subcategory of $\Sch_R$ consisting of schemes which are smooth and
essentially of finite type over $R$. 
We shall denote the product of $X, Y \in \Sch_R$ by $X \times_R Y$.
We shall let ${\Sch_R}/{\zar}$ denote the Grothendieck site on $\Sch_R$
given by the Zariski topology.

\subsection{Commutative monoids }\label{sec:Not}
Let $\N$ denote the set of all non-negative integers.
Let $M$ be a commutative and finitely generated 
{\sl monoid} (a semi-group with unit element). 
Recall that $M$ is said to be {\sl cancellative} if given
$a, a', b \in M$, we have $a+ b = a' + b \Rightarrow a = a'$.
One says that $M$ is {\sl torsion-free} if given any integer $c \ge 1$
and $a, b \in M$, one has $ca = cb \Rightarrow a = b$. 
We say that $M$ is {\sl reduced} if it has no non-trivial units
(i.e, $a = 0$ in $M$ if and only if $a+ a' = 0$ for some $a' \in M$).
It is well known that a finitely generated commutative monoid
is cancellative, torsion-free and reduced if and only if
it is isomorphic to a submonoid of $\N^r$ for some integer $r \ge 1$
(see, for instance, \cite[Theorem~3.11]{RGS}). 

Let $M$ be a commutative and cancellative monoid and let $gp(M)$ denote
its group completion. Recall that $M$ is said to be {\sl semi-normal}
if for every $a \in gp(M)$ with $2a, 3a \in M$, we have $a \in M$. 
One says that $M$
is {\sl normal} if every element $a \in gp(M)$ lies in $M$ if and only if
$ma \in M$ for some integer $m \ge 1$. 
We say that $M$ is $c$-divisible for some positive integer $c$ if for any 
$x  \in M$, there exists $y  \in M$ for which
$cy = x$. 
In this text, we shall assume all monoids to be commutative, 
finitely generated, cancellative, torsion-free, reduced and normal. 
We shall use the shorthand decoration {\sl nice} for such monoids in the sequel.

\subsection{Gubeladze's monoid algebra}\label{sec:GMA}
Our principal interest in this text is to study the algebraic $K$-theory
of the following monoid algebra, considered by Gubeladze \cite[\S~1]{Gub-2}.

For any integer $r \ge 1$, let $\{e_1, \cdots , e_r\}$ denote the standard
ordered basis of $\N^r$ as a monoid. Here, $e_i$ is the vector whose
$i$th coordinate is one and others are all zero.
Let $M$ denote the submonoid of $\N^4$ generated by the vectors
$\{v_1, \cdots, v_4\}$, where
\[
v_1 = e_1 + e_3, \ v_2 = e_1 + e_4, \ v_3 = e_2 + e_3, \ \ \mbox{and} \ \ 
v_4 = e_2 + e_4.
\]

We shall use the term {\sl Gubeladze's monoid} for the one given above
for the rest of this text.
One can check that $M$ is not divisible by any integer $c > 1$.
Indeed, if $M$ is $c$-divisible for integer $c > 1$, then we can write
$v_1 = c(a_1v_1 + a_2v_2 + a_3v_3 + a_4v_4)$, where $a_i \in \N$.
But this implies that $c(a_1+ a_2) = 1$, which is not possible.

It is also easy to check that for any ring $R$, the monoid algebra $R[M]$
is isomorphic to the monomial $R$-algebra 
$R[z_1z_3, z_2z_4, z_1z_4, z_2z_3]$, 
considered as the subalgebra of the polynomial algebra $R[z_1, \cdots, z_4]$
via the isomorphism $v = (m_1, \cdots , m_4) \mapsto z^{m_1}_1 \cdots z^{m_4}_4$. 
There is in turn an isomorphism of $R$-algebras
\begin{equation}\label{eqn:GMA-0}
\phi_R: \frac{R[x_1, \cdots , x_4]}{(x_1x_2 - x_3x_4)} \xrightarrow{\simeq} 
R[z_1z_3, z_2z_4, z_1z_4, z_2z_3];
\end{equation}
\[
\phi_R(x_1) = z_1z_3, \ \phi_R(x_2) = z_2z_4, \ \phi_R(x_3) = z_1z_4, \
\phi_R(x_4) = z_2z_3.
\]   

Let $\epsilon_R : R[M] \surj R$ denote the augmentation map of the monoid 
algebra $R[M]$ so that the composite $R \to R[M] \to R$ is identity and
let $\bar{\fm}_R = (x_1, \cdots , x_4)$ denote the augmentation ideal.
It is well known (and not so difficult to check) that the monoid $M$ above
is nice. In particular, $R[M]$ is an integral domain if $R$ is so and
$R[M]$ is normal if $R$ is a normal integral domain (see, for instance,
\cite[Theorem~4.40]{BG}). 
On the other hand, it follows from \cite[Proposition~4.45]{BG} that
$R[M]$ is not regular even if $R$ is. We shall study the singularity of $R[M]$
shortly. From now on, $R[M]$ will always denote the monoid algebra
described in ~\eqref{eqn:GMA-0}. We shall often refer to it as
Gubeladze's monoid algebra.

\subsection{The desingularization of $\Spec(R[M])$ and
its geometry}\label{sec:Geom-RM}
Let $R$ be a regular ring. We observed in \S~\ref{sec:Intro} 
that $R[M]$ is the homogeneous 
coordinate ring for the Segre embedding of $E_R = \P^1_R \times_R \P^1_R$ 
inside $\P^3_R$. We shall use this geometric interpretation 
in our study of the monoid algebra $R[M]$.
Letting $\P^3_R = \Proj_R(R[x_1, \cdots , x_4])$, this closed
embedding is defined by the graded surjection $R[x_1, \cdots , x_4] \surj 
R[M]$ of ~\eqref{eqn:GMA-0}. We shall fix this embedding 
$\iota_R: E_R \inj \P^3_R$ 
throughout the text. We shall let
$\Proj_R(R[z_1, z_2])$ be the first factor and 
$\Proj_R(R[z_3, z_4])$ the second factor of $E_R$.

If we let $X_R = \Spec(R[M])$, then it is the affine cone over
$E_R$ whose vertex $P_R$ is the closed subscheme $\Spec(R) \inj \Spec(R[M])$ 
defined by the augmentation ideal of $R[M]$. 
If we let $\pi_R: V_R \to X_R$ denote the blow-up along the vertex $P_R$, then
we get a commutative diagram
\begin{equation}\label{eqn:Main-diagram}
\xymatrix@C1pc{
E_R \ar@{^{(}->}[r] \ar[d]_{{p}_{E_R}} \ar@/^1.5pc/[rrr]^{{\rm id}} & 
E_{n,R} \ar@{^{(}->}[r] \ar[d] & V_R 
\ar[d]^{\pi_R} \ar[r]^{\wt{p}_R} & E_R \ar[d]^{p_{E_R}} \\
P_R \ar@{^{(}->}[r] & P_{n,R} \ar@{^{(}->}[r] & X_R \ar[r]_{p_{X_R}} & \Spec(R),}
\end{equation}
where $P_{n,R}$ is the closed subscheme of $X_R$ defined by the ideal
$\bar{\fm}^n_R$ and $E_{n,R} = \pi^*_R(P_{n,R})$ so that the left and the middle
squares are Cartesian. The arrows $p_{X_R}$ and $p_{E_R}$ are the 
structure maps. It is well known that $\wt{p}_R: V_R \to E_R$ is the line
bundle associated to the invertible sheaf $\sO_E(1)$ (with respect to
the Segre embedding) on $E_R$ and the inclusion of the exceptional divisor
$E_R \inj V_R$ for $\pi_R$ is same as the 0-section of $\wt{p}_R$.  
We let ${p}_R$ be the restriction of $\wt{p}_R$ to $E_{n,R}$.
We let $\sI_R$ denote the ideal sheaf of the exceptional divisor $E_R$ so that
$\sI^n_R$ defines $E_{n,R} \inj V_R$ for $n \ge 1$. In this case, one also
knows that ${\sI^n_R}/{\sI^{n+1}_R} \simeq \sO_{E_R}(n)$ for every $n \ge 0$.

We let $V'_R$ denote the blow-up of $\A^4_R$ at the origin, 
$P'_{n, R} = \Spec({R[x_1, \cdots , x_4]}/{\fm^n_R})$, 
where $\fm_R$ is the defining ideal
of the origin in $\A^4_R$ and let $E'_{n, R}$ be the inverse image of
$P'_{n,R}$ under the blow-up $\pi': V'_R \to \A^4_R$.
We let $E'_R = \pi'^{-1}(P_R) \simeq \P^3_R$.
We have a commutative diagram
\begin{equation}\label{eqn:Blow-up-2}
\xymatrix@C1pc{
& X_R \times E_R  \ar[dd] \ar[rr] \ar[dl]_{q_R} & & X_R \ar[dd] \\
E_R \ar[dd] & & V_R \ar[ul] \ar[ur]_{\pi_R} \ar[ll]^/-.2cm/{\wt{p}_R} \ar[dd] 
& \\
& \A^4_R \times E'_R \ar[rr] \ar[dl]_{q'_R} & & \A^4_R \\ 
\P^3_R & & V'_R. \ar[ul] \ar[ur]_{\pi'_R} \ar[ll]^{p'_R} &}
\end{equation}

In this diagram, all vertical arrows are closed immersions,
$\pi_R$ and $\pi'_R$ are blow-up maps, $\wt{p}_R$ and $p'_R$ are 
line bundle projections, $q_R$ and $q'_R$ are projections, 
and all squares and triangles commute.
The lower square on the front is Cartesian. We note that
\[
V'_R = \Proj_{R[x_1,\cdots, x_4]}({R[x_1, \cdots ,x_4][y_1, \cdots, y_4]}/I), \ \
V_R =  \Proj_{R[M]}({R[M][y_1, \cdots, y_4]}/J),
\]
where
\begin{equation}\label{eqn:Blow-up-3}
I = (x_1y_2 - x_2y_1, 
x_1y_3-x_3y_1, x_1y_4-x_4y_1, x_2y_3-x_3y_2, x_2y_4-x_4y_2, x_3y_4-
x_4y_3)
\end{equation}
\[
\mbox{and} \ \ J = I + (y_1y_2 - y_3y_4).
\]

We also have a commutative diagram
\begin{equation}\label{eqn:Blow-up-4}
\xymatrix@C1pc{
& E_{n,R} \ar[dl]_{{p}_R}
\ar[r] \ar[d] & P_{n,R} \times E_R \ar[r] \ar[d] & P_{n,R} \ar[d] \\
E_R & V_R \ar[r] \ar[l]^{\wt{p}_R} & X_R \times E_R \ar[r] & X_R,}
\end{equation}
with Cartesian squares for each $n \ge 1$, 
where the vertical and the left horizontal
arrows are closed immersions, and the right horizontal arrows are projections.

In terms of the coordinate rings over an affine open subset of $\P^3_R$ of the 
form $U_1 = \{y_1 \neq 0\}$, ~\eqref{eqn:Blow-up-4} is translated into
a commutative diagram
\begin{equation}\label{eqn:Blow-up-5}
\xymatrix@C1pc{
\frac{R[x_1, \cdots , x_4]}{(x_1x_2-x_3x_4)} \ar[r] \ar[d] &
\frac{R[x_1, \cdots , x_4, y_3,y_4]}{(x_1x_2-x_3x_4)} \ar[r]^-{\alpha} \ar[d] &
R[x_1, y_3, y_4] \ar[d] & R[y_3, y_4] \ar[l] \ar[dl] \\
\frac{R[x_1, \cdots , x_4]}{(x_1x_2-x_3x_4, \fm^n_R)} \ar[r] &
\frac{R[x_1, \cdots , x_4, y_3,y_4]}{(x_1x_2-x_3x_4, \fm^n_R)} 
\ar[r]^-{\ov{\alpha}} & \frac{R[x_1, y_3, y_4]}{(x^n_1)},}
\end{equation}
where one checks from ~\eqref{eqn:Blow-up-3} that
\[
\alpha(x_1) = x_1, \ \alpha(x_2) = y_3y_4x_1, \  \alpha(x_3) = y_3x_1, \
\alpha(x_4) = y_4x_1, \ \alpha(y_3) = y_3 \ \ \mbox{and} \ \
\alpha(y_4) = y_4.
\]

If we replace $E_R$ by $\P^3_R$, then the coordinate rings of the
terms in ~\eqref{eqn:Blow-up-4} lead to a commutative diagram
\begin{equation}\label{eqn:Blow-up-6}
\xymatrix@C1pc{
R[x_1, \cdots , x_4] \ar[r] \ar[d] &
R[x_1, \cdots , x_4, y_2, y_3,y_4] \ar[r]^-{\alpha'} \ar[d] &
R[x_1, y_2, y_3, y_4] \ar[d] & R[y_2, y_3, y_4] \ar[l] \ar[dl] \\
\frac{R[x_1, \cdots , x_4]}{\fm^n_R} \ar[r] &
\frac{R[x_1, \cdots , x_4, y_2, y_3,y_4]}{\fm^n_R} 
\ar[r]^-{\ov{\alpha '}} & \frac{R[x_1, y_2,y_3, y_4]}{(x^n_1)},}
\end{equation}
where $\alpha'(x_1) = x_1, \alpha'(x_2) = y_2x_1, \alpha'(x_3) = y_3x_1$ 
and $\alpha'(x_4) = y_4x_1$.
The diagram ~\eqref{eqn:Blow-up-6} maps onto ~\eqref{eqn:Blow-up-5}
via the transformation $y_2 \mapsto y_3y_4$.

We shall use the following notations for various coordinate rings
over the open subset $\{y_1 \neq 0\}$ of $\P^3_R$.
We let 
\begin{equation}\label{eqn:Blow-up-7}
Q = {R[x_1, \cdots, x_4]}/{(x_1x_2-x_3x_4)}, \ 
Q_n = {R[x_1, \cdots, x_4]}/{(x_1x_2-x_3x_4, \fm^n_R)}, 
\end{equation}
\[
A = R[y_3, y_4], \ \ R_n = {R[x_1]}/{(x^n_1)} \ \ \mbox{and} \ \
A_n = {A[x_1]}/{(x^n_1)} = R_n[y_3, y_4]
\]
so that $\Spec(Q) = X_R$ and $\Spec(Q_n) = P_{n,R}$.
One checks using ~\eqref{eqn:Blow-up-5} and ~\eqref{eqn:Blow-up-6} that 
$A_n = {Q_n[y_3, y_4]}/{J_R}$, where
\[
J'_R = \Ker(\ov{\alpha'}) =
(x_2 -y_3y_4x_1, x_3 - y_3x_1, x_4 - y_4x_1, x_2y_3 - x_3y_2,
x_2y_4 - x_4y_2, x_3y_4-x_4y_3) \ \mbox{and} 
\]
\begin{equation}\label{eqn:H0-Omega-1}
J_R = \Ker({\ov{\alpha}}) = \frac{J'_R + (y_2-y_3y_4)}{(y_2-y_3y_4)} = 
(x_2 -y_3y_4x_1, x_3 - y_3x_1, x_4 - y_4x_1).
\end{equation}

\subsection{Cohomology of $E_R$}\label{sec:E-Coh}
We shall need to use the following result about the cohomology groups
on $E_R$ repeatedly in this text. Let $A_R = R[z_1, z_2] = \bigoplus_{i \ge 0}
A_{R,i}$ and $A'_R = R[z_3, z_4] = \bigoplus_{i \ge 0} A'_{R,i}$ be the homogeneous
coordinate rings of the two factors of $E_R$ and
let $\phi_R: R[M] \xrightarrow{\simeq} A \boxtimes_R A' = 
\bigoplus_{i \ge 0} A_{R,i} \otimes_R A'_{R,i}$. 

\begin{lem}\label{lem:Coh-main}
Let $R$ be a Noetherian regular ring containing $\Q$ and let $n \in \Z$. 
Then the following hold. 
\begin{enumerate}
\item
$\Omega^2_{{E_R}/R}(n) \simeq \sO_{E_R}(n-2)$. 
\item
$H^0(E_R, \sO_{E_R}(n)) \simeq A_{R,n} \otimes_R A'_{R,n}$.
In particular, $H^0(E_R, \sO_{E_R}(n)) = 0$ for $n < 0$. 
\item
$H^1(E_R, \sO_{E_R}(n)) = 0$.
\item
\[
H^2(E_R, \sO_{E_R}(n)) \simeq \left\{
\begin{array}{ll}
A_{R,2-n} \otimes_R A'_{R,2-n} & \mbox{if $n < -1$} \\
0 & \mbox{if $n \ge -1$.}
\end{array}
\right.
\]
\item
$H^0(E, \Omega^1_{{E_R}/{R}}(n)) \simeq (A_{n-2} \otimes_{R} A_n)^{\oplus 2}$.
\item
\[
\begin{array}{lll}
H^1(E_R, \Omega^1_{{E_R}/{R}}(n)) & \simeq & (A_{R,-n} \otimes_{R} A'_{R,n}) \oplus 
(A_{R,n-2} \otimes_{R} A'_{R,-2-n}) \\
& &  \hspace*{3cm} \oplus \\
& & (A_{R,-2-n} \otimes_{R} A'_{R,n-2}) \oplus (A_{R,n} \otimes_{R} A'_{R,-n}).
\end{array}
\]
\item
$H^2(E_R, \Omega^1_{{E_R}/R}(n)) \simeq (A_{R,-n}\otimes_{R} A'_{R,-2-n}) \oplus
(A_{R,-2-n} \otimes_{R} A'_{R,-n})$.
\item
$H^0(E_R, \Omega^2_{{E_R}/R}(n)) \simeq A_{R,n-2} \otimes_R A'_{R,n-2}$ and
$H^1(E_R, \Omega^2_{{E_R}/R}(n)) = 0$. 
\item
\[
H^2(E_R, \Omega^2_{{E_R}/R}(n)) \simeq \left\{
\begin{array}{ll}
A_{R,-n} \otimes_R A'_{R,-n} & \mbox{if $n \le 0$} \\
0 & \mbox{if $n > 0$.}
\end{array}
\right.
\]
\end{enumerate}
\end{lem} 
\begin{proof} 
The key point in this proof (and elsewhere in the text) is the observation
that our monoid algebra $R[M]$ is the base change of the monoid algebra
$\Q[M]$ via the inclusion $\Q \subset R$. Furthermore, it is evident
from ~\eqref{eqn:Main-diagram} that all schemes in the diagram
as well as all squares are obtained by the (flat) 
base change of a similar set of schemes and squares over $\Q$. 
Since $\Omega^i_{{E_{\Q}}/{\Q}}(n) \otimes_{\Q} R \xrightarrow{\simeq}
\Omega^i_{{E_R}/R}(n)$ for all $i \ge 0$ and $n \in \Z$, we can apply
\cite[Theorem~12]{Kemf} to reduce the proof 
to the case when $R = \Q$. 
We therefore assume this to be the case and drop the subscript $\Q$ from all 
notations. 

If $p_1, p_2: E \to \P^1$ denote two projections, then one knows
that $\sO_{E}(n) \simeq p^*_1(\sO_{\P^1}(n)) \otimes_{E} p^*_2(\sO_{\P^1}(n)) 
:= \sO_{\P^1}(n) \boxtimes \sO_{\P^1}(n) := \sO_E(n,n)$.
We also know that $\Omega^1_{{E}/{\Q}} \simeq (\Omega^1_{{\P^1}/{\Q}} \boxtimes 
\sO_{\P^1}) \oplus (\sO_{\P^1} \boxtimes \Omega^1_{{\P^1}/{\Q}})$.
In particular, we get 
\[
\Omega^1_{{E}/{\Q}}(n) \simeq 
(\Omega^1_{{\P^1}/{\Q}}(n) \boxtimes \sO_{\P^1}(n)) \oplus (\sO_{\P^1}(n) 
\boxtimes \Omega^1_{{\P^1}/{\Q}}(n)) \simeq \sO_E(n-2, n) \oplus \sO_E(n, n-2).
\]
We similarly get $\Omega^2_{{E}/{\Q}}(n) \simeq \sO_E(n-2)$
and $\Omega^{> 2}_{{E}/{\Q}}(n) = 0$, which proves (1).

The proof of (2), (3) and (4) now follows immediately from 
the K{\"u}nneth formula \cite[Theorems~13, 14]{Kemf} and 
\cite[Theorem~5.1]{Hart}.
Another application of the  K{\"u}nneth formula gives
$H^0(E, \Omega^1_{{E}/{\Q}}(n)) \simeq 
(A_{n-2} \otimes_{\Q} A'_n) \oplus (A_n \otimes_{\Q} A'_{n-2}) 
\simeq (A_{n-2} \otimes_{\Q} A_n)^{\oplus 2}$.
Next we have
\[
\begin{array}{lll}
H^1(E, \sO_E(n-2,n)) & \simeq &
H^1(\P^1, \sO_{\P^1}(n-2)) \otimes_{\Q} H^0(\P^1, \sO_{\P^1}(n)) \\
& & \hspace*{3cm} \oplus  \\
& & H^0(\P^1, \sO_{\P^1}(n-2)) \otimes_{\Q} H^1(\P^1, \sO_{\P^1}(n)) \\
& \simeq & (A_{-n} \otimes_{\Q} A'_{n}) \oplus (A_{n-2} \otimes_{\Q} A'_{-2-n}).
\end{array}
\]

Similarly, we have $H^1(E, \sO_E(n,n-2)) \simeq (A_{-2-n} \otimes_{\Q} A'_{n-2})
\oplus (A_n \otimes_{\Q} A'_{-n})$.
We conclude that
\[
H^1(E, \Omega^1_{{E}/{\Q}}(n)) \simeq (A_{-n} \otimes_{\Q} A'_{n}) \oplus 
(A_{n-2} \otimes_{\Q} A'_{-2-n}) \oplus (A_{-2-n} \otimes_{\Q} A'_{n-2})
\oplus (A_n \otimes_{\Q} A'_{-n}).
\]
\[
H^2(E, \Omega^1_{{E}/{\Q}}(n)) \simeq (A_{-n}\otimes_{\Q} A'_{-2-n}) \oplus
(A_{-2-n} \otimes_{\Q} A'_{-n}).
\]
We have thus proven (5), (6) and (7). The last two assertions
follow from (1) $\sim$ (4). 
\end{proof}

One can check by an elementary calculation, using ~\eqref{eqn:GMA-0} and
\lemref{lem:Coh-main}, that the map
$H^0(\P^3_R, \sO_{\P^3_R}(n)) \to H^0(E_R, \sO_{E_R}(n))$ is surjective
for all $n \ge 0$. In view of \cite[Exc.~II.5.14]{Hart},
this gives a geometric proof of the fact that $R[M]$ is normal.
Combining \lemref{lem:Coh-main} with \cite[Lemma~1.3]{Krishna-1},
we subsequently conclude that $R[M]$ is also Cohen-Macaulay.

\section{Vanishing theorems}\label{sec:R-reg}
Let $k$ be a field of characteristic zero. 
Let $k[M]$ be Gubeladze's monoid algebra and consider the 
diagram~\eqref{eqn:Main-diagram}, where we drop the subscript $k$ from all
notations. But we shall continue to use the subscript $R$ if it is 
different from $k$.

One of our main tools to study the algebraic $K$-theory of $R[M]$
(for a $k$-algebra $R$) will be the use of Chern class maps from $K$-theory to 
Hochschild and cyclic homology. We refer the reader to \cite{Loday} for 
details on this subject. For any subring $l \subset k$ and any map 
$A \to B$ of commutative $k$-algebras, 
Loday also defines the relative Hochschild homology $HH^l_*(A, B)$ over $l$
as the homology groups of the chain complex ${\rm Cone}(C_{\bullet}(A) \to
C_{\bullet}(B))[-1]$, where $C_{\bullet}(A)$ denotes the Hochschild complex
of $A$ over $l$. 
For an ideal $I$ of $A$, $HH^l_*(A, I)$ is the relative
Hochschild homology of $A$ and $A/I$.

Given an ideal $I \subset A$ such that
$I = IB$, one defines the double relative Hochschild homology
$HH^l(A,B,I)$ as the homology of the complex 
${\rm Cone}(C_{\bullet}(A,I) \to C_{\bullet}(B,I))[-1]$ over $l$. 
The relative and double relative 
cyclic homology are defined in a similar way by taking the cones over the total 
cyclic complexes. These are denoted by $HC^l_*(A,I)$ and
$HC^l_*(A, B,I)$, respectively. Recall from \cite[\S~4.5, 4.6]{Loday}
that the Hochschild and cyclic homology as well as their relative and double
relative companions have $\lambda$-decomposition in terms of the
Andr{\'e}-Quillen homology and the de Rham cohomology.

There are Chern class maps (Dennis trace maps) \cite[8.4.3]{Loday}
$K_i(A) \to HH^l_i(A)$ and by functoriality of 
fibrations of $K$-theory spectra and Hochschild homology, one also has Chern 
class maps from relative $K$-theory to relative Hochschild homology which are 
compatible with long exact sequence of relative $K$-theory and Hochschild 
homology. These Chern classes induce similar Chern character maps from
relative (for nilpotent ideals) and double relative $K$-theory 
to the corresponding cyclic homology.

The Hochschild and cyclic homology
(and their relative and double relative companions)
are defined for any $X \in \Sch_k$ and they coincide with the above definitions 
for affine schemes (see, for instance, \cite{Weibel-1}).
We let ${{\sH}{\sH}}^l_{i, X}$ (resp. ${{\sH}{\sC}}^l_{i, X}$) denote the Zariski 
sheaf on $X$ associated to the 
presheaf $U \mapsto HH^l_i(U)$ (resp. $U \mapsto HC^l_i(U)$). 
We define the sheaves of relative 
and double relative cyclic homology in analogous way. As just remarked, the
stalks of these sheaves are the (relative, double relative) 
Hochschild and cyclic homology of the associated local rings. 
The functoriality of the
$\lambda$-decomposition gives rise to the corresponding decomposition
for the sheaves as well \cite{Weibel-1}. 
The resulting Hodge pieces of these sheaves will be 
denoted by ${{\sH}{\sH}}^{l, (j)}_{i, X}$ and ${{\sH}{\sC}}^{l, (j)}_{i, X}$ 
(resp. ${{\sH}{\sH}}^{l, (j)}_{i, (X,W)}$,
${{\sH}{\sH}}^{l, (j)}_{i, (X,W,Z)}$ and ${{\sH}{\sC}}^{l, (j)}_{i, (X,W)}$,
${{\sH}{\sC}}^{l, (j)}_{i, (X,W,Z)}$).

Given $X \in \Sch_k$, we shall denote its Andre-Quillen homology sheaves
relative to $l \subset k$ by $\sD^{(q)}_p(X/l)$, where
$p, q \ge 0$. We shall freely use various standard facts about 
Andr{\'e}-Quillen, Hochschild and cyclic homology of rings 
without a specific reference. They can all be found in \cite{Loday}.

In this paper, all Hochschild, cyclic homology and sheaves of differential
forms without the mention of the base ring 
$l \subset k$ will be assumed to be considered over $l = \Q$. 

\subsection{$K$-theory, Hochschild and cyclic homology relative to $R$}
\label{sec:Not-R}
Let $R$ be a regular $k$-algebra and set $S = \Spec(R)$.
In this paper, we shall use the following notations to take care of
various (co)homology and homotopy groups relative to the $k$-algebra $R$.

Given a presheaf $\sF$ of abelian groups on $\Sch_k$, we let 
${{_R}\sF}$ denote the sheaf on 
${\Sch_k}/{\zar}$ associated to the presheaf $X \mapsto \sF(X \times_k S)$.
If $\sF$ is a presheaf of spectra, we define the presheaf of spectra ${{_R}\sF}$
in a similar way. The idea behind the introduction of the sheaves 
${{_R}\sF}$ is that it often allows us to compute various functors (e.g.,
$K$-theory, Hochschild homology and cyclic homology) on $R$-algebras in terms
of sheaf cohomology on $\Sch_k$ (e.g., see Theorems~\ref{thm:TT-gen}
and ~\ref{thm:MT-General}).

Given a morphism $f:X \to Y$ in $\Sch_k$, the relative
$K$-theory spectrum ${_RK(Y,X)}$ is defined to be 
the homotopy fiber of the map of spectra $f^*:{_RK(Y)} \to {_RK(X)}$.
Given any $X \in \Sch_k$, we shall let ${_R\wt{K}(X)}$ denote the relative
$K$-theory spectrum ${{_R}K(\Spec(k), X)}$. 
Similarly, the relative Hochschild homology ${_RHH_*(A,B)}$ will be the
homology of ${\rm Cone}(C_{\bullet}(A\otimes_k R) \to
C_{\bullet}(B \otimes_k R))[-1]$ over the base field $\Q$.

Associated to the commutative diagram~\eqref{eqn:Main-diagram}, we shall
use the following notations. For any closed subscheme $Z \subset V$ containing
$E$, we shall denote the relative $K$-theory $_RK(Z,E)$ by $_R\wt{K}(Z)$.
We shall use similar notations for the relative $K$-theory, Hochschild and
cyclic homology sheaves on $Z$. The relative $K$-theory (resp. Hochschild and
cyclic homology) $_RK_*(P_n, P)$ (resp. $_RHH_*(P_n, P)$ and 
$_RHC_*(P_n, P)$) will be denoted by $_R\wt{K}_*(P_n)$ 
(resp. $_R\wt{HH}_*(P_n)$ and $_R\wt{HC}_*(P_n)$). 
We shall write $\Ker({_R\Omega^*_{Z}} \to {_R\Omega^*_{E}})$ as
${_R\wt{\Omega}^*_{Z}}$. The notation ${_R\wt{\Omega}^*_{P_n}}$ will
denote $\Ker({_R\Omega^*_{P_n}} \to {_R\Omega^*_{P}})$.
These are all over the base field $\Q$.

Given an affine map $f: X \to Y$ in $\Sch_k$, we let 
${_R\sHH^{(j)}_{*, X/Y}}$ be the Zariski sheaf on $Y$ whose value at any 
affine open $U = \Spec(A) \subset Y$ is $HH^{A \otimes_k R, (j)}_*(B \otimes_k R)$,
where $\Spec(B) = f^{-1}(U)$. We define the relative cyclic and 
Andr{\'e}-Quillen homology sheaves and the sheaves of relative differential 
forms ${_R\sHC^{(j)}_{*, X/Y}}, {_R\sD^{(q)}_p(X/Y)}$ and
${_R\Omega^*_{X/Y}}$ in an analogous way.
In the above notations, we shall often drop the subscript $R$ if $R = k$.

\subsection{Pro-objects in abelian categories}\label{sec:Pro-ab}
By a pro-object in an abelian category $\sC$, we shall mean a sequence 
$\{A_1 \xleftarrow{\alpha_1} A_2 \xleftarrow{\alpha_2} \cdots \}$ of objects 
in $\sC$. It will be formally denoted by $\{A_i\}_i$. 
A morphism $f: \{A_i\}_i \to \{B_j\}_j$ in the category ${\rm pro}\sC$ of
pro-objects in $\sC$ is an element of the set
${\underset{j}\varprojlim} {\underset{i}\varinjlim} \
\Hom_{\sC}(A_i, B_j)$. 
In particular, such a morphism $f$ is 
same as giving a function $\lambda: \N^{+} \to \N^{+}$ and a morphism
$f_i : A_{\lambda(i)} \to B_{i}$ in $\sC$ for each $i \ge 1$
such that for any $j \ge i$, there is some $l \ge \lambda(i),  \lambda(j)$
so that the diagram
\begin{equation}\label{eqn:Ind-Obj}
\xymatrix@C1pc{
A_l \ar[r] \ar[dr] & A_{\lambda(j)} \ar[r]^-{f_j} &  B_j \ar[d] \\
& A_{\lambda(i)} \ar[r]_-{f_i} & B_i}
\end{equation}
commutes in $\sC$.
We shall call such a morphism to be {\sl strict} if  
$\lambda(j) \ge \lambda(i)$ and $l = \lambda(j)$ for every $j \ge i$.
In this paper, we shall use only the strict morphisms.
Moreover, except in ~\eqref{eqn:ext-d}, the function $\lambda$ will actually be
identity.
It is known that ${\rm pro}\sC$ is an abelian category.
The following description of kernels and cokernels in ${\rm pro}\sC$
is elementary.

\begin{lem}\label{lem:Pro-ker-Co}
Given a strict morphism $f: \{A_i\}_i \to \{B_i\}_i$ in ${\rm pro}\sC$,
one has ${\rm Ker}(f) = \{{\rm Ker}(f_i)\}_i$
and ${\rm Coker}(f) = \{{\rm Coker}(f_i)\}_i$.
In particular, a sequence of strict morphisms 
\begin{equation}\label{eqn:pro-exact}
\{A_i\}_i \xrightarrow{\lambda}
\{B_i\}_i \xrightarrow{\gamma} \{C_i\}_i 
\end{equation}
is exact in ${\rm pro}\sC$ if the sequence
$A_{\lambda \circ \gamma (i)} \to B_{\gamma(i)} \to C_i$
is exact for every $i \ge 1$. 
\end{lem}

We refer the reader to \cite[Appendix~4]{AM} for these facts about pro-objects
in abelian categories.

\subsection{Some vanishing results}\label{sec:Key-vanish}
In this section, we prove some vanishing theorems for the
cohomology of relative cyclic homology sheaves. Given a quasi-coherent
sheaf $\sF$ on $V$ (resp. on $E_{n+1}$), we shall denote the sheaf
$\sF \otimes_V \sO_{E_n}$ (resp. $\sF \otimes_{E_{n+1}} \sO_{E_n}$) by
$\sF|_{E_n}$. We shall write ${\rm Ker} ({_R\Omega^m_{{E_{n+1}}/l}} \to 
{_R\Omega^m_{{E_{n}}/l}})$ as ${_R\Omega^m_{(E_{n+1}, E_n)/l}}$ for any
subring $l \subset k$.
These notations will be used throughout the text.
We shall use the following elementary computations for exterior products.

\begin{lem}\label{lem:Elm-alg-0}
Let $A$ be a ring and let 
\begin{equation}\label{eqn:Elm-alg-1}
0 \to M' \to M \to M'' \to 0
\end{equation}
be a short exact sequence of $A$-modules, and let $m \ge 1$ be an integer.
Then, there exists a finite filtration $\{F^{\bullet}{\wedge^m} M\}$ of
$\wedge^m M$ by $A$-submodules and a surjection
$\wedge^{i}M' \otimes_A \wedge^{m-i}M'' 
\surj {F^i\wedge^m M}/{F^{i+1}\wedge^m M}$. This surjection is an
isomorphism if ~\eqref{eqn:Elm-alg-1} splits.
\end{lem}
\begin{proof}
We can define a decreasing finite filtration 
on ${\wedge}^m M$ by defining $F^j {\wedge}^m M$ to be the $A$-submodule
generated by the forms of the type
\[\left\{a_1 \wedge \cdots \wedge a_r | a_{i_1}, \cdots , a_{i_j} \in
M' \ {\rm for \ some} \ 1 \le i_1 \le \cdots \le i_j \le m \right\}.\]
Then we have
\[
{\wedge}^m M = F^0 {\wedge}^m M \supseteq \cdots \supseteq
F^m {\wedge}^m M \supseteq F^{m+1} {\wedge}^m M = 0
\] 
and it is easy to check that  for $0 \le j \le m$, the map
\[
{\beta}^j_i : {\wedge}^j M' {\otimes}_A {\wedge}^{m-j} M \to
F^j {\wedge}^m M,
\]
\[
{\beta}^j_i \left((a_1 \wedge \cdots \wedge a_j) \otimes
(b_1 \wedge \cdots \wedge b_{m-j})\right) =
a_1 \wedge \cdots \wedge a_m \wedge b_1 \wedge \cdots \wedge b_{m-j}\]
descends to a surjective map of quotients
\begin{equation}\label{eqn:descend} 
{\beta}^j_i :
{\wedge}^j M' {\otimes}_A {\wedge}^{m-j} M'' \surj
\frac {F^j {\wedge}^m M} {F^{j+1} {\wedge}^m M}.
\end{equation}
One also checks easily that this map is an isomorphism
if $M \simeq M' \oplus M''$. We leave out the details 
as an exercise.
\end{proof}

\begin{lem}\label{lem:ext-pro}
Let $A$ be a ring and let $f:\{M'_n\}_n \to \{M_n\}_n$ be a map
of pro-$A$-modules, induced by a compatible system of surjective
maps $f_n:M'_n \surj M_n$ of $A$-modules. Assume that $f$ is an isomorphism.
Then the induced map $\wedge^r f: \{{\wedge}^rM'_n\}_n \to 
\{{\wedge}^rM_n\}_n$ is also an isomorphism for every $r \ge 1$.
\end{lem}
\begin{proof}
Let $M''_n = \Ker(f_n)$ so that the sequence
\[
0 \to M''_n \to M'_n \to M_n \to 0
\]
is a short exact sequence of $A$-modules for every $n \ge 0$.
This yields a short exact sequence
\begin{equation}\label{eqn:ext-pro-0}
0 \to E^r_n \to T^r(M'_n) \to T^r(M_n) \to 0,
\end{equation}
where $T^r(-)$ denotes the functor of tensor power of $A$-modules
and $E^r_n$ is the submodule of $T^r(M'_n)$, generated by the
tensors $\{a_1 \otimes \cdots \otimes a_r| \ a_i \in
M''_n \ {\rm for \ some} \ 1 \le i \le r\}$.

Since $\{M''_n\}_n = 0$ by hypothesis, it is clear that $\{E^r_n\}_n = 0$
and this yields $\{T^r(M'_n)\}_n \xrightarrow{\simeq} \{T^r(M_n)\}_n$.
Since $\{\wedge^r(N_n)\}_n$ is a canonical direct summand of 
$\{T^r(N_n)\}_n$ for any pro-$A$-module $\{N_n\}_n$ (see 
\cite[Lemma~2.2]{Krishna-2}), the lemma follows.
\end{proof}

\begin{lem}\label{lem:Pro-trivial}
Let $k$ be a field of characteristic zero and let $Y \subset X$ be a closed
embedding of regular, Noetherian $k$-schemes of finite Krull dimension.
Then, $\{HC^{k,(i)}_n(rY, Y)\}_r = 0$ for $0 \le i < n$.
\end{lem}
\begin{proof}
This is \cite[Proposition~3.3]{KM}. The only point to be noted is
that this result is stated in the cited reference for $k = \Q$ case.
However, the proof works verbatim for any field $k$ containing $\Q$.
The reason is that it uses only \cite[Theorem~3.23]{Morrow-1}
and \cite[Corollary~9.9.3]{Weibel-2}, both of which are valid in the 
general situation.
\end{proof}

We now return to the study of our monoid algebra.
Since our monoid algebra is defined over $\Q$, we can write
\begin{equation}\label{eqn:vanish-cyclic-2-2}
{_R\Omega^m_{E_n}} \simeq \Omega^m_{E_n \otimes_k R}
\simeq \Omega^m_{E_{n, \Q} \otimes_{\Q} R}
\simeq \stackrel{m}{\underset{i = 0}\oplus} \Omega^i_{E_{n, \Q}} \otimes_{\Q} 
\Omega^{m-i}_R \simeq  \stackrel{m}{\underset{i = 0}\oplus}
 \Omega^i_{{E_n}/k} \otimes_{k} \Omega^{m-i}_R. 
\end{equation}
In particular, we have
\begin{equation}\label{eqn:vanish-cyclic-2-3} 
{_R\wt{\Omega}^i_{E_n}} = {\rm Ker}({_R\Omega^i_{E_n}} \surj 
{_R\Omega^i_{E}}) \simeq \stackrel{m}{\underset{i = 0}\oplus} 
\wt{\Omega}^i_{{E_n}/k} \otimes_{k} 
\Omega^{m-i}_R \ \simeq \
\stackrel{m}{\underset{i = 0}\oplus} \wt{\Omega}^i_{E_{n, \Q}} 
\otimes_{\Q} \Omega^{m-i}_R \ {\rm for} \ m \ge 0 \ {\rm and} \ n \ge 1.
\end{equation}

The following is our key lemma to prove some of the vanishing results. In
its proof (and elsewhere), we shall use the following notation: 
for any $k$-algebra $B$, ${B[v]dv}/{(v^m,v^ndv)}$ will denote 
${B[v]dv}/{(v^ndv)}\otimes_B {B[v]}/{(v^m)}$ for 
$m, n\geq 0$.

\begin{lem}\label{lem:Key}
For integers $m \ge 0$ and $n \ge 1$, there is a short exact sequence
\begin{equation}\label{eqn:Key-0}
0 \to \Omega^m_{E/k}(n) \to \Omega^m_{(E_{n+1}, E_n)/k} \to \Omega^{m-1}_{E/k}(n) 
\to 0.
\end{equation}
\end{lem}
\begin{proof}
The $m = 0$ case immediately follows from the
exact sequence
\[
0 \to  {\sI^{n}}/{\sI^{n+1}} \to \sO_{E_{n+1}} \to \sO_{E_n} \to 0
\]
and the isomorphism ${\sI^{n}}/{\sI^{n+1}} \simeq \sO_E(n)$.

For $m \ge 1$, we consider
the exact sequence
\begin{equation}\label{eqn:vanish-omega-1-2} 
0 \to {\sI^{n}}/{\sI^{n+1}} \xrightarrow{d} \Omega^1_{{E_{n+1}}/k}|_{E_n} \to 
\Omega^1_{{E_n}/k} \to 0.
\end{equation}

We need to justify the injectivity of the map $d$.
But this is a local verification.  
We shall use the notations of ~\eqref{eqn:Blow-up-7} for this purpose 
(with $R = k$). We thus have $E = \Spec(A) = \Spec(k[y_3,y_4]), \ 
V = \Spec(A[x_1]), \ k_n = {k[x_1]}/{(x^n_1)}$ and 
$E_n = \Spec({k[x_1,y_3,y_4]}/{(x^n_1)})$. 
We then get
$\Omega^m_{{A_n}/k} \ \simeq \ \stackrel{m}{\underset{i = 0}\oplus} 
\Omega^i_{A/k} \otimes_{k} \Omega^{m-i}_{{k_n}/k}$.

We also have for $j \ge 0$:
\begin{equation}\label{eqn:vanish-omega-1-3} 
\begin{array}{lll}
(\Omega^i_{A/k} \otimes_{k} \Omega^{m-i}_{{k_n}/k}) \otimes_{A_n} (x^j_1)A_n
& \simeq & 
(\Omega^i_{A/k} \otimes_{k} \Omega^{m-i}_{{k_n}/k}) \otimes_{A_n} 
(A \otimes_k {(x^j_1)}/{(x^n_1)}) \\
& \simeq & (\Omega^i_{A/k} \otimes_{k} \Omega^{m-i}_{{k_n}/k}) \otimes_{A_n} 
(A \otimes_k (k_n \otimes_{k_n} {(x^j_1)}/{(x^n_1)})) \\
& \simeq & (\Omega^i_{A/k} \otimes_{k} \Omega^{m-i}_{{k_n}/k}) \otimes_{A_n}
(A_n \otimes_{k_n} {(x^j_1)}/{(x^n_1)}) \\ 
& \simeq & \Omega^i_{A/k} \otimes_{k} (\Omega^{m-i}_{{k_n}/k} 
\otimes_{k_n} {(x^j_1)}/{(x^n_1)}). \\
\end{array}
\end{equation}

Since the map ${(z^n)}/{(z^{n+1})} \xrightarrow{d} \Omega^1_{{k_{n+1}}/k}|_{k_n}
\simeq {k[z]dz}/{(z^{n}, z^ndz)}$ is clearly injective (here we use
characteristic zero), it follows that
the composite map $A \otimes_k {(z^n)}/{(z^{n+1})} \xrightarrow{d}
\Omega^1_{{A_{n+1}}/k}|_{A_n} \surj A \otimes_k \Omega^1_{{k_{n+1}}/k}|_{k_n}$ is 
injective.
This shows that the map $d$ in ~\eqref{eqn:vanish-omega-1-2} is injective.

Since ${\sI^{n}}/{\sI^{n+1}} \simeq \sO_E(n)$ is an invertible sheaf on $E$,
it follows from ~\lemref{lem:Elm-alg-0} that there exists an exact sequence
\begin{equation}\label{eqn:vanish-omega-1-4} 
\Omega^{m-1}_{{E_n}/k} \otimes_{E_n} \sO_E(n) \xrightarrow{id \wedge d} 
\Omega^{m}_{{E_{n+1}}/k}|_{E_n} \to \Omega^m_{{E_n}/k} \to 0.
\end{equation}

\noindent
{\bf Claim~1:}
The map $\id \wedge d$ factors as
$\Omega^{m-1}_{{E_n}/k} \otimes_{E_n} \sO_E(n) \surj \Omega^{m-1}_{E/k} \otimes_E 
\sO_E(n) \inj \Omega^{m}_{{E_{n+1}}/k}|_{E_n}$.

Using the local calculation above ~\eqref{eqn:vanish-omega-1-3}, we get
$\Omega^m_{{A_{n+1}}/k}|_{A_n} \ \simeq \ \stackrel{m}{\underset{i = 0}\oplus} 
\Omega^i_{A/k} \otimes_{k} (\Omega^{m-i}_{{k_{n+1}}/k}|_{k_n})$
and
$\Omega^{m-1}_{{A_n}/k} \otimes_{A_n} {(x^n_1)}/{(x^{n+1}_1)}A_n
\ \simeq \ \stackrel{m-1}{\underset{i = 0}\oplus} 
\Omega^i_{A/k} \otimes_{k} (\Omega^{m-1-i}_{{k_n}/k} 
\otimes_{k_n} {(x^n_1)}/{(x^{n+1}_1)})$.
We thus need to show that the map
${(x^n_1)}/{(x^{n+1}_1)} \xrightarrow{d} \Omega^{1}_{{k_{n+1}}/k}|_{k_n}$ is 
injective
and the map $\Omega^{m-1}_{{k_n}/k} \otimes_{k_n} {(x^n_1)}/{(x^{n+1}_1)}
\xrightarrow{id \wedge d} \Omega^{m}_{{k_{n+1}}/k}|_{k_n}$ is zero for all
$m \ge 2$. But the injectivity at $m = 1$ is already shown above,
and the vanishing at $m \ge 2$ follows immediately from the 
fact that $\Omega^m_{{k_n}/k} = 0$ for $m \ge 2, n \ge 1$. This proves the claim.

Using Claim~1 and the splitting of the inclusion $E \inj E_n$, we obtain
a short exact sequence for $m \ge 0$:
\begin{equation}\label{eqn:vanish-omega-1-5} 
0 \to \Omega^{m-1}_{E/k}(n)
\to \wt{\Omega}^{m}_{{E_{n+1}}/k}|_{E_n} \to \wt{\Omega}^m_{{E_n}/k} \to 0.
\end{equation}

Using this exact sequence, we get a commutative diagram with exact rows
\begin{equation}\label{eqn:vanish-omega-1-6} 
\xymatrix@C1pc{
0 \ar[r] & \Omega^m_{(E_{n+1}, E_n)/k} \ar[r] \ar@{->>}[d] &  
\wt{\Omega}^{m}_{{E_{n+1}}/k} 
\ar[r] \ar@{->>}[d] & \wt{\Omega}^m_{{E_n}/k} \ar[r] \ar@{=}[d] & 0 \\
0 \ar[r] & \Omega^{m-1}_{E/k}(n) \ar[r] & \wt{\Omega}^{m}_{{E_{n+1}}/k}|_{E_n} 
\ar[r] & \wt{\Omega}^m_{{E_n}/k} \ar[r] &  0.}
\end{equation}
This yields an exact sequence
\begin{equation}\label{eqn:vanish-omega-1-7} 
\Omega^m_{{E_{n+1}}/k} \otimes_{E_{n+1}} {\sI^n}/{\sI^{n+1}} \to
\Omega^m_{(E_{n+1}, E_n)/k} \to \Omega^{m-1}_{E/k}(n) \to 0.
\end{equation}

\noindent
{\bf Claim~2:} The map $\Omega^m_{{E_{n+1}}/k} \otimes_{E_{n+1}} 
{\sI^n}/{\sI^{n+1}} \to
\Omega^m_{(E_{n+1}, E_n)/k}$ has a factorization
\[
\Omega^m_{{E_{n+1}}/k} \otimes_{E_{n+1}} {\sI^n}/{\sI^{n+1}} \surj
\Omega^m_{E/k} \otimes_{E/k} {\sI^n}/{\sI^{n+1}} \inj \Omega^m_{(E_{n+1}, E_n)/k}.
\]

The case $m = 0$ is obvious and so we assume $m \ge 1$, where we can replace
$\Omega^m_{(E_{n+1}, E_n)/k}$ by $\Omega^m_{E_{n+1}/k}$.
We can check this locally. Using ~\eqref{eqn:vanish-omega-1-3}, we need to show
that the map $\Omega^{m}_{{k_{n+1}}/k} \otimes_{k_{n+1}} {(x^n_1)}/{(x^{n+1}_1)} \to 
\Omega^m_{{k_{n+1}}/k}$ is injective for $m = 0$ and zero for $m \ge 1$.
The injectivity for $m =0$ is obvious and we have
$\Omega^m_{{k_{n+1}}/k} = 0$ for $m \ge 2$. So we only need to check 
$m = 1$ case. Here, the map in question is the canonical map
$\frac{k[x_1]dx_1}{(x^{n+1}_1, x^n_1dx_1)} \otimes_{{k_{n+1}}} 
\frac{(x^n_1)}{(x^{n+1}_1)} \to 
\frac{k[x_1]dx_1}{(x^{n+1}_1, x^n_1dx_1)}$.
However, this is same as the map
$\frac{k[x_1]dx_1}{(x_1, x^n_1dx_1)} \otimes_k \frac{(x^n_1)}{(x^{n+1}_1)}  \to
\frac{k[x_1]dx_1}{(x^{n+1}_1, x^n_1dx_1)}$, which is turn, is same as
the map $\frac{(x^n_1)}{(x^{n+1}_1)} \to
\frac{k[x_1]dx_1}{(x^{n+1}_1, x^n_1dx_1)}$, given by $x^{n+i}_1 \mapsto 
x^{n+i}_1dx_1$. 
But this is clearly zero and we get Claim~2.
The lemma follows from Claim~2 and ~\eqref{eqn:vanish-omega-1-7}.
\end{proof}

\begin{lem}\label{lem:vanish-omega-1}
For any set of integers $m \ge 0$ and $i, n \ge 1$, we have 
$H^i(E_n, {_R\wt{\Omega}^m_{E_n}})= 0$.
\end{lem}
\begin{proof}
By ~\eqref{eqn:vanish-cyclic-2-3}, the lemma is equivalent to showing
that $H^i(E_{n, \Q}, \wt{\Omega}^m_{E_{n, \Q}})= 0$ for $m \ge 0$ and 
$i, n \ge 1$.
So we can assume $k = R = \Q$. The statement is now clearly true for $n =1$
by \lemref{lem:Coh-main}.

For $n \ge 2$ and $i \ge 1$, we use the 
top row of ~\eqref{eqn:vanish-omega-1-6} and an induction on $n \ge 1$. 
This reduces us to showing that $H^i(E_n, \Omega^m_{(E_{n+1}, E_n)}) = 0$
for $m \ge 0$ and $i, n \ge 1$. Using \lemref{lem:Key}, it suffices to show that
$H^i(E, \Omega^m_E(n)) = 0$ for $m \ge 0$ and $i, n \ge 1$. 
But this follows from \lemref{lem:Coh-main}.
\end{proof}

\begin{lem}\label{lem:vanish-cyclic-1}
For any integers $m \ge 0$ and $i \ge 1$, 
we have $\{H^i(E_n, {_R\tHC_{m, E_n}})\}_n = 0$.
\end{lem}
\begin{proof}
Using the Hodge decomposition of the cyclic homology sheaf, we have
a natural decomposition of Zariski sheaves
${_R\tHC_{0, E_n}} \simeq {\sI}/{\sI^n}\otimes_k R$ and
${_R\tHC_{m, E_n}} \simeq \ \stackrel{m}{\underset{j = 1}\oplus} 
{_R\tHC^{(j)}_{m, E_n}}$
for $m \ge 1$ (see \cite[Theorem~4.6.7]{Loday}).
In particular, we have $H^*(E_n, {_R\tHC_{0, E_n}}) \simeq 
H^*(E_n, {\sI}/{\sI^n})\otimes_k R$. 

We first prove the case $m = 0$. We show that
$H^i(E_n, {\sI}/{\sI^n}) = 0$ for all $i, n \ge 1$.
The statement is clearly true for $n =1$.
In general, we use induction and the short exact sequence
\begin{equation}\label{eqn:vanish-cyclic-2-0*}
0 \to \frac{\sI^{n}}{\sI^{n+1}} \to  \frac{\sI}{\sI^{n+1}} \to 
\frac{\sI}{\sI^{n}} \to 0.
\end{equation} 
This reduces the problem to showing that $H^i(E, \sO_E(n)) = 0$ for
all $i, n \ge 1$. But this follows from \lemref{lem:Coh-main}. 

When $m \ge 1$, it follows from \lemref{lem:Pro-trivial} that
$\{H^i(E_n, {_R\tHC^{(j)}_{m, E_n}})\}_n = 0$ for $1 \le j \le m-1$
and $i \ge 1$ (here we use the regularity of $R$). We thus have to show that
$\{H^i(E_n, {_R\tHC^{(m)}_{m, E_n}})\}_n = 0$ for $i \ge 1$. 

Using the isomorphism ${_R\sHC^{(m)}_{m, E_n}} \simeq 
\frac{{_R\Omega^m_{E_n}}}{d({_R\Omega^{m-1}_{E_n}})}$
(see \cite[Theorem~4.5.12]{Loday}) and the splitting of the inclusion
$E \subset E_n$, we get ${_R\tHC^{(m)}_{m, E_n}} \simeq
\frac{{_R\wt{\Omega}^m_{E_n}}}{d({_R\wt{\Omega}^{m-1}_{E_n}})}$.
Equivalently, there is an exact sequence
\begin{equation}\label{eqn:vanish-cyclic-1-0}
{_R\wt{\Omega}^{m-1}_{E_n}} \xrightarrow{d_{m-1}} {_R\wt{\Omega}^m_{E_n}} \to
{_R\tHC^{(m)}_{m, E_n}} \to 0.
\end{equation}

In particular, a diagram with exact row:
\begin{equation}\label{eqn:vanish-cyclic-1-1}
\xymatrix@C1pc{
& & H^2(E_n, {_R\wt{\Omega}^{m-1}_{E_n}}) \ar@{->>}[d] \\
H^1(E_n, {_R\wt{\Omega}^m_{E_n}}) \ar[r] & H^1(E_n, {_R\tHC^{(m)}_{m, E_n}}) \ar[r]
& H^2(E_n, d({_R\wt{\Omega}^{m-1}_{E_n}}))}  
\end{equation}
and a surjection $H^2(E_n, {_R\wt{\Omega}^m_{E_n}}) \surj
H^2(E_n, {_R\tHC^{(m)}_{m, E_n}})$. 

Since $H^2$-functor is right exact on the category of Zariski sheaves on
$E_n$, it suffices to show that
$H^i(E_n, {_R\wt{\Omega}^m_{E_n}}) = 0$ for each $m \ge 0$ and $i, n \ge 1$.
But this follows from \lemref{lem:vanish-omega-1}.
\end{proof}

\subsection{The map 
$H^0(E_n, {_R\wt{\Omega}^m_{E_n}}) \to 
H^0(E_n, {_R\tHC^{(m)}_{m, E_n}})$}\label{sec:Ker-d}
It is easy to check from the local description of the projection
$\wt{p}: V \to E$ that its restriction $p: E_n \to E$ is a finite
morphism. Let ${_R\Omega^m_{{E_n/E}}}$ be the Zariski sheaf on $E$
as defined in ~\S~\ref{sec:Not-R}.
It is clear from its definition that there is an exact sequence of
Zariski sheaves of $E$:
\begin{equation}\label{eqn:Ker-d-0}
{_R\Omega^1_{E}} \otimes_{E} p_*(\sO_{E_n}) \to p_*({_R\Omega^1_{E_n}}) \to
p_*({_R\Omega^1_{{E_n/E}}}) \to 0.
\end{equation}

We wish to show that the first arrow in this sequence is injective.
It is enough to check this locally on $E$.
Over an affine open subset of $E$ of the form $\Spec(k[y_3,y_4])$,
we let $A = R[y_3,y_4], \ B_n = {\Q[x_1]}/{(x^n_1)}$ and $\wt{B}_n =
{(x_1)}/{(x^n_1)}$.
The above exact sequence over this open subset is of the form
\begin{equation}\label{eqn:Ker-d-1}
\Omega^1_A \otimes_A A_n \to \Omega^1_{A_n} \to \Omega^1_{{A_n}/A} \to 0.
\end{equation}

On the other hand, we have $A_n = A \otimes_{\Q} B_n$ and hence
\begin{equation}\label{eqn:Ker-d-1*}
\Omega^m_{A_n} \simeq \stackrel{m}{\underset{i =0}\oplus}
\Omega^i_A \otimes_{\Q} \Omega^{m-i}_{B_n} \simeq
(\Omega^m_A \otimes_{\Q} B_n) \oplus (\Omega^{m-1}_A \otimes_{\Q} \Omega^1_{B_n})
\ \ \mbox{and}
\end{equation}
\[
\wt{\Omega}^m_{A_n} \simeq (\Omega^m_A \otimes_{\Q} \wt{B}_n) \oplus 
(\Omega^{m-1}_A \otimes_{\Q} \Omega^1_{B_n}).
\]

Using the identifications
$\Omega^1_{{A_n}/A} \simeq A \otimes_{\Q} \Omega^1_{B_n}$ and
$\Omega^1_A \otimes_A A_n \simeq \Omega^1_A \otimes_{\Q} B_n$, it follows
that ~\eqref{eqn:Ker-d-1}, and hence ~\eqref{eqn:Ker-d-0}
yields a short exact sequence
\begin{equation}\label{eqn:Ker-d-2}
0 \to {_R\Omega^1_{E}} \otimes_{E} p_*(\sO_{E_n}) \to p_*({_R\Omega^1_{E_n}}) \to
p_*({_R\Omega^1_{{E_n/E}}}) \to 0.
\end{equation}

In particular, we get a locally split short exact sequence
\begin{equation}\label{eqn:Ker-d-3}
0 \to {_R\Omega^1_{E}} \otimes_{E} p_*({\sI}/{\sI^n}) \to 
p_*({_R\wt{\Omega}^1_{E_n}}) \to p_*({_R\Omega^1_{{E_n/E}}}) \to 0.
\end{equation}

We let $d_m: {_R\wt{\Omega}^m_{E_n}} \to {_R\wt{\Omega}^{m+1}_{E_n}}$ 
be the differential map over $\Q$ and let 
${_R\wt{\Omega}^{m, m+1}_{E_n}} = \Ker(d_m)$.
Using that $p_*$ is left exact on the category of sheaves of abelian groups,
it follows from ~\eqref{eqn:Ker-d-1*} that there is an
exact sequence
\begin{equation}\label{eqn:Ker-d-4}
0 \to {\Ker}({_R\Omega^1_{E}} \otimes_{E} p_*({\sI}/{\sI^n}) \to 
p_*({_R\wt{\Omega}^2_{E_n}})) \to p_*({_R\wt{\Omega}^{1, 2}_{E_n}}) \to 
p_*({_R\Omega^1_{{E_n/E}}}) \to 0.
\end{equation}

\vskip .3cm

\noindent
{\bf Claim~1:} The map $ p_*({_R\wt{\Omega}^{1, 2}_{E_n}}) \to 
p_*({_R\Omega^1_{{E_n/E}}})$ is an isomorphism. 

This is equivalent to the first term of ~\eqref{eqn:Ker-d-4}
being zero. We can check this locally, and in this case,
we need to show that the map
$\Omega^{1}_A \otimes_{\Q} \wt{B}_n \to 
(\Omega^2_A \otimes_{\Q} \wt{B}_n) \oplus 
(\Omega^1_A \otimes_{\Q} {\Omega}^1_{B_n})$ is injective. 
But this follows from the fact that the composite map
$\Omega^{1}_A \otimes_{\Q} \wt{B}_n \to 
(\Omega^2_A \otimes_{\Q} \wt{B}_n) \oplus 
(\Omega^1_A \otimes_{\Q} {\Omega}^1_{B_n}) \to 
\Omega^1_A \otimes_{\Q} {\Omega}^1_{B_n}$ is an isomorphism.
This proves the claim.

\vskip .3cm

\noindent
{\bf Claim~2:} For $m \ge 2$, the map
${_R\Omega^{m-1}_E} \otimes_E  p_*({\sI}/{\sI^n}) \xrightarrow{d_{m-1}}
p_*({_R\wt{\Omega}^{m, m+1}_{E_n}})$ is an isomorphism.

Using the left exactness of $p_*$, it suffices to show that
the map 
${_R\Omega^{m-1}_E} \otimes_E  p_*({\sI}/{\sI^n}) \xrightarrow{d_{m-1}}
\Ker(p_*({_R\wt{\Omega}^{m}_{E_n}}) \to p_*({_R\wt{\Omega}^{m+1}_{E_n}}))$
is an isomorphism. We can again check this locally on $E$. With respect to the
local description of ~\eqref{eqn:Ker-d-1*}, we need to check that the
map
$(\Omega^m_A \otimes_{\Q} \wt{B}_n) \oplus 
(\Omega^{m-1}_A \otimes_{\Q} \Omega^1_{B_n})
\xrightarrow{d_{m}} (\Omega^{m+1}_A \otimes_{\Q} \wt{B}_n) \oplus 
(\Omega^{m}_A \otimes_{\Q} \Omega^1_{B_n})$ has kernel 
$\Omega^{m-1}_A \otimes_{\Q} \wt{B}_n$. 

First of all, the composite map
$\Omega^{m-1}_A \otimes_{\Q} \wt{B}_n \xrightarrow{d_{m-1}}
(\Omega^m_A \otimes_{\Q} \wt{B}_n) \oplus (\Omega^{m-1}_A \otimes_{\Q} 
\Omega^1_{B_n}) \to (\Omega^{m-1}_A \otimes_{\Q} \Omega^1_{B_n})$ is an 
isomorphism. It follows that the first map is injective.

We next let $\alpha$ be the composite of $d_m$ with the projection
$(\Omega^{m+1}_A \otimes_{\Q} \wt{B}_n) \oplus 
(\Omega^{m}_A \otimes_{\Q} \Omega^1_{B_n}) \to \Omega^{m+1}_A \otimes_{\Q} 
\wt{B}_n$. 
We then have 
\begin{equation}\label{eqn:Ker-d-5}
\begin{array}{lll}
\Ker(d_m) & = & \Ker(\Ker(\alpha) \to (\Omega^{m+1}_A \otimes_{\Q} \wt{B}_n)
\oplus (\Omega^{m}_A \otimes_{\Q} \Omega^1_{B_n}))\\
& = & \Ker((\Omega^{m,m+1}_A \otimes_{\Q} \wt{B}_n) \oplus
(\Omega^{m-1}_A \otimes_{\Q} \Omega^1_{B_n})  \xrightarrow{d_{m}} 
\Omega^{m}_A \otimes_{\Q} \Omega^1_{B_n}). 
\end{array}
\end{equation} 

We now let $F^m = (\Omega^{m,m+1}_A \otimes_{\Q} \wt{B}_n) \oplus
(\Omega^{m-1}_A \otimes_{\Q} \Omega^1_{B_n})$ and consider the commutative diagram
\begin{equation}\label{eqn:Ker-d-6}
\xymatrix@C1pc{
& & \Omega^{m-1}_A \otimes_{\Q} \wt{B}_n \ar[dr]^{id \otimes d_0} 
\ar[d]^{d_{m-1}} & & \\
0 \ar[r] & \Omega^{m,m+1}_A \otimes_{\Q} \wt{B}_n
\ar[r] \ar@{^{(}->}[d]_{id \otimes d_0} & F^m \ar[r] \ar[d]^{d_m} & 
\Omega^{m-1}_A \otimes_{\Q} \Omega^1_{B_n} \ar[r] & 0 \\
& \Omega^{m}_A \otimes_{\Q} {\Omega}^1_{B_n} \ar@{=}[r] &
\Omega^{m}_A \otimes_{\Q} {\Omega}^1_{B_n}. & &}
\end{equation}

Given any $w \in \Omega^{m-1}_A$ and $v \in \Omega^1_{B_n}$, we have
$d_{m}(w \otimes v) = d_m(w) \otimes v = d_m(w) \otimes d_0(v')$ 
for some $v' \in \wt{B}_n$.
Since $d_m(w) \in \Omega^{m,m+1}_A$, it follows that
the map $\Omega^{m-1}_A \otimes_{\Q} \Omega^1_{B_n} \to 
{\rm Coker}(id \otimes d_0)$ is zero. Since the diagonal arrow 
$id \otimes d_0$ on the top right is an isomorphism, we conclude that
the middle vertical sequence is exact. This proves the claim.

\vskip .3cm

The final result of this section is the following.

\begin{lem}\label{lem:surjection-H0}
For any integer $m \ge 0$, the map
$H^0(E_n, {_R\wt{\Omega}^m_{E_n}}) \to H^0(E_n, {_R\tHC^{(m)}_{m, E_n}})$ is 
surjective.
\end{lem}
\begin{proof}
The lemma is equivalent to showing that the map
\begin{equation}\label{eqn:surjection-H0-0}
H^0(E, p_*({_R\wt{\Omega}^m_{E_n}})) \to H^0(E, p_*({_R\tHC^{(m)}_{m, E_n}}))
\end{equation}
is surjective.
The stalk of $p_*({_R\wt{\Omega}^m_{E_n}})$ at any point $P' \in E$
is $\Ker(\Omega^m_{R\otimes_{k} \sO_{E_n, P'}} \to \Omega^m_{R\otimes_{k} \sO_{E, P'}})$.
Similarly, the stalk of ${_R\tHC^{(m)}_{m, E_n}}$ at $P'$ is
$\Ker(HC^{(m)}_{m}(R\otimes_{k} \sO_{E_n, P'}) \to 
HC^{(m)}_{m}(R\otimes_{k} \sO_{E, P'}))$.
It follows from ~\eqref{eqn:vanish-cyclic-1-0} that the sequence
\[
\wt{\Omega}^{m-1}_{R\otimes_{k} \sO_{E_n, P'}} 
\xrightarrow{d_{m-1}} \wt{\Omega}^{m}_{R\otimes_{k} \sO_{E_n, P'}} 
\to \wt{HC}^{(m)}_{m}(R\otimes_{k} \sO_{E_n, P'}) \to 0
\]
is exact.
We conclude from this that there is an exact sequence of sheaves on $E$:
\begin{equation}\label{eqn:surjection-H0-1}
p_*({_R\wt{\Omega}^{m-1}_{E_n}}) \xrightarrow{d_{m-1}} 
p_*({_R\wt{\Omega}^{m}_{E_n}}) \xrightarrow{\ov{d_{m-1}}} 
p_*({_R\tHC^{(m)}_{m, E_n}}) \to 0.
\end{equation}

Using ~\eqref{eqn:surjection-H0-1}, \lemref{lem:vanish-omega-1}
and the exactness of $p_*$ on the category of quasi-coherent sheaves,
the associated sequence of cohomology groups yields
a commutative diagram
\begin{equation}\label{eqn:extra-explanation}
\xymatrix@C1pc{
H^0(E, p_*({_R\wt{\Omega}^m_{E_n}})) \ar[r]^{\ov{d_{m-1}}} &
H^0(E, p_*({_R\tHC^{(m)}_{m, E_n}})) \ar[r] & H^1(E, {\rm Ker}(\ov{d_{m-1}}))
\ar@{^{(}->}[d] \\
& & H^2(E, {\rm Ker}(d_{m-1})),}
\end{equation}
where the top row is exact.
It suffices therefore to show that $H^2(E, \Ker(d_{m-1})) = 0$.
Since $\Ker(d_0) = 0$ (as one can check from ~\eqref{eqn:Ker-d-1*}), it suffices
to show that $H^2(E, \Ker(d_{m})) = 0$ for all $m \ge 1$.

If $m =1$, it follows from Claim~1 and ~\eqref{eqn:Ker-d-3}
that $H^2(E, p_*({_R\wt{\Omega}^{1}_{E_n}})) \surj 
H^2(E, \Ker(d_{m}))$. We thus conclude from \lemref{lem:vanish-omega-1}.

If $m \ge 2$, we reduce the problem to showing that
$H^2(E, {_R\Omega^{m-1}_E} \otimes_E  p_*({\sI}/{\sI^n})) = 0$ in view of 
Claim~2.
We have a filtration 
\begin{equation}\label{eqn:extra}
0 \subset {\sI^{n-1}}/{\sI^n} \subset \cdots \subset {\sI^2}/{\sI^n} 
\subset {\sI}/{\sI^n}.
\end{equation}

Since ${_R\Omega^{m-1}_E}$ is a locally free $\sO_E$-module (of
possibly infinite rank), we get a similar filtration of
${_R\Omega^{m-1}_E} \otimes_E  p_*({\sI}/{\sI^n})$ after
tensoring the above filtration with ${_R\Omega^{m-1}_E}$.
Using this filtration and induction on $n \ge 1$, it suffices to show that
$H^2(E, {_R\Omega^{m-1}_E}(n)) = 0$ for every $n \ge 1$.
However, we have $H^2(E, {_R\Omega^{m-1}_E}(n)) \simeq 
H^2(E, \Omega^{m-1}_E(n)) \otimes_k R$ and $H^2(E, \Omega^{m-1}_E(n)) = 0$
for $n \ge 1$ by \lemref{lem:Coh-main}.
This finishes the proof.
\end{proof}

\begin{cor}\label{cor:surjection-H0-ext}
For any integer $m \ge 0$, the map
$H^0(E_n, {_R\wt{\Omega}^m_{{E_n}/k}}) \to H^0(E_n, {_R\tHC^{k,(m)}_{m, E_n}})$ is 
surjective.
\end{cor}
\begin{proof}
This is a direct consequence of \lemref{lem:surjection-H0} and
the decomposition ~\eqref{eqn:vanish-cyclic-2-3}.
\end{proof}

\section{Analysis of $H^0(E_n, {_R\tHC_{m, E_n}})$}
\label{sec:Global-sec}
Let $p: E_{n} \to E$ be the projection map
(see ~\eqref{eqn:Main-diagram}). This is a finite morphism.
For $p, q \ge 0$, let ${_R\sD^{(q)}_p({E_n}/E)}$ be the Zariski sheaf on $E$ 
defined in ~\S~\ref{sec:Not-R}. 
Following the classical notation, we shall write 
${_R\sD^{(1)}_p({E_n}/E)}$ as ${_R\sD_p({E_n}/E)}$.

\begin{lem}\label{lem:AQ-0}
The map ${_R\sD_1({E_n}/k)} \to {_R\sD_1({E_n}/E)}$ is an isomorphism.
\end{lem}
\begin{proof}
By \cite[3.5.5.1]{Loday}, we have the exact sequence of Zariski sheaves
on $E$:
\begin{equation}\label{eqn:AQ-0-0}
{_R\sD_1({E}/k)}\otimes_{E} p_*(\sO_{E_n}) \to
{_R\sD_1({E_n}/k)} \to {_R\sD_1({E_n}/E)} \to
{_R\Omega^1_{E/k}} \otimes_E p_*(\sO_{E_n})
\end{equation}
\[
\hspace*{8cm} \to p_*({_R\Omega^1_{{E_n}/k}})
\to p_*({_R\Omega^1_{{E_n}/E}}) \to 0.
\]

Since $E_R \in \Sm_R$, the first term of this exact sequence vanishes
by \cite[Theorem~3.5.6]{Loday}.  It suffices therefore, to show that
the map ${_R\Omega^1_{E/k}} \otimes_E p_*(\sO_{E_n}) \to 
p_*({_R\Omega^1_{{E_n}/k}})$
is injective. But this follows from a simpler version of ~\eqref{eqn:Ker-d-2}.
\end{proof}

\begin{lem}\label{lem:AQ-1}
There are canonical isomorphisms 
$p_*({\sI^{n+1}}/{\sI^{2n}}) \otimes_k R \simeq {_R\sD_1({E_n}/E)} \simeq$ \\
${_R\sHH^{(1)}_{2,{E_n}/{E}}}$ of Zariski sheaves on $E$.
\end{lem}
\begin{proof}
The second isomorphism is from  \cite[Proposition~4.5.13]{Loday}.
So we only need to verify the first isomorphism.
Since the relative Andr{\'e}-Quillen homology sheaf ${_R\sD_1({V}/E)}$ 
vanishes, we have an exact sequence of Zariski sheaves
\begin{equation}\label{eqn:AQ-1-0}
0 \to {_R\sD_1({E_n}/E)} \to p_*({_R{\sI^n}/{\sI^{2n}}}) \xrightarrow{d} 
p_*({_R\Omega^1_{V/E}|_{E_n}}) \to p_*({_R\Omega^1_{{E_n}/E}}) \to 0.
\end{equation}
Hence, we need to show that the canonical inclusion
${\sI^{n+1}}/{\sI^{2n}} \inj {\sI^n}/{\sI^{2n}}$ is identified with
the kernel of the map ${\sI^n}/{\sI^{2n}} \xrightarrow{d} 
\Omega^1_{V/E}|_{E_n}$.

Since this is a local calculation with $R = k$, 
we can assume $E = \Spec(A)$, where
$A = k[y_3,y_4]$. Following the notations of
~\eqref{eqn:Blow-up-5}, we have $\Omega^1_{V/E}|_{E_n} = {A[x_1]dx_1}/{(x^n_1)}$ 
and it is easy to check that the kernel of the map
\[
d: \frac{(x^n_1)}{(x^{2n}_1)} \to \frac{A[x_1]dx_1}{(x^n_1)}
\]
is precisely ${(x^{n+1}_1)}/{(x^{2n}_1)}$.
\end{proof}

\begin{lem}\label{lem:H0-Omega}
The restriction map ${_R\Omega^1_{{E_n}/{P_n}}} \to {_R\Omega^1_{{E}/{k}}}$
is an isomorphism. In particular, we have  
$H^0(E_n, {_R\Omega^1_{{E_n}/{P_n}}}) = 0$.
\end{lem}
\begin{proof}
We let $F_n = P_n \times E$ and let $\sJ = \Ker(\sO_{F_n} \surj \sO_{E_n})$
so that there is a commutative diagram (with exact top row) of Zariski sheaves 
on $E$:
\begin{equation}\label{eqn:H0-Omega-0}
\xymatrix@C1pc{
{_R{\sJ}/{\sJ^2}} \ar[r]^-{d} &  {_R\Omega^1_{{F_n}/{P_n}}} \otimes_{F_n}
\sO_{E_n} \ar[r] \ar@{->>}[d]_{\eta_1}  & {_R\Omega^1_{{E_n}/{P_n}}} \ar[r] 
\ar@{->>}[d]^{\eta_2}  & 0 \\
& {_R\Omega^1_{E/k}} \ar@{=}[r] & {_R\Omega^1_{E/k}}.}
\end{equation}

To prove the first assertion of the lemma, it suffices therefore to show that 
the map
${_R{\sJ}/{\sJ^2}} \xrightarrow{d} \Ker(\eta_1)$ is surjective.
This is again a local calculation and so we replace $E_R$ by
$\Spec(A)$, where $A = R[y_3,y_4]$ and follow the notations of
~\eqref{eqn:Blow-up-5} and ~\eqref{eqn:Blow-up-7}. 

In the local notations, we have ${_R\Omega^1_{{F_n}/{P_n}}} \otimes_{F_n}
\sO_{E_n} = \Omega^1_{{Q_n[y_3, y_4]}/{Q_n}} \otimes_{Q_n[y_3, y_4]} A_n
\simeq A_ndy_3 \oplus A_n dy_4$
and ${_R\Omega^1_{E/k}} \simeq Ady_3 \oplus Ady_4$.
It follows that $\Ker(\eta_1) = x_1A_ndy_3 \oplus x_1A_ndy_4$.
So we need to show that the $A_n$-linear (and hence $Q_n[y_3, y_4]$-linear)
map ${J_R}/{J^2_R} \xrightarrow{d} x_1A_ndy_3 \oplus x_1A_ndy_4$ is surjective, 
where $J_R = \Ker({\ov{\alpha}}) =  
(x_2 -y_3y_4x_1, x_3 - y_3x_1, x_4 - y_4x_1)$.

Now, any element of the right hand side of this map is of the form
$w = fx_1 dy_3 + gx_1dy_4$ with $f, g \in A_n$. If let $f'$ (resp. $g'$) denote
a lift of $f$ (resp. $g$) via $\ov{\alpha}$, then we get
\begin{equation}\label{eqn:H0-Omega-2}
d(f'(y_3x_1-x_3)) {=}^1 fd(y_3x_1-x_3) {=}^2 fx_1 dy_3,
\end{equation}
where ${=}^1$ follows from the fact that $d$ is $Q_n[y_3, y_4]$-linear
and ${=}^2$ follows from the fact that $d$ is induced by the
derivation  $d: Q_n[y_3, y_4] \to \Omega^1_{{Q_n[y_3, y_4]}/{Q_n}}$
over $Q_n$ and the fact that $\{x_1, \cdots , x_4\} \subset Q_n$.
Since $f'(y_3x_1-x_3) \in J_R$ by ~\eqref{eqn:H0-Omega-1}, we conclude that
$fx_1dy_3$ lies in the image of $d$. In the same way, we get
$d(g'(y_4x_1-x_4)) = gx_1dy_4$ and $g'(y_4x_1-x_4) \in J_R$.
This proves the claimed surjectivity and hence the first part of the
lemma.

The second part follows directly from the first part and
\lemref{lem:Coh-main}, which implies that $H^0(E_n, {_R\Omega^1_{E/k}})
\simeq H^0(E, {_R\Omega^1_{E/k}}) = 0$.
\end{proof}

Before we prove our next key result, we present the following local computation
of some Andr{\'e}-Quillen homology.

\begin{lem}\label{lem:local-AQ}
In the notations of local coordinate rings in ~\eqref{eqn:Blow-up-7},
we have $D_1({A_n}/{Q_n}) \simeq {G_n}/{J_n}$, where
$G_n$ is the submodule of the free $A_n$-module 
$A_ndx_2 \oplus A_ndx_3 \oplus A_ndx_4$ of the form
$\{fdx_2 + (x^{n-1}_1g - fy_4)dx_3 + (x^{n-1}_1h - fy_3)dx_4| f,g, h \in A_n\}$
and $J_n = (x_1dx_2 - y_4x_1dx_3 - y_3x_1dx_4, d(\fm^n_R))$.
\end{lem}
\begin{proof}
It is easy to check using ~\eqref{eqn:Blow-up-5} that there is a
commutative diagram
\begin{equation}\label{eqn:Key-Key-3}
\xymatrix@C1pc{
R_n \ar[r] \ar[d] & Q_n \ar[d] \ar[dr] & & \\
A_n \ar[r] \ar@/_1.5pc/[rr]_{id} & Q_n[y_3,y_4] \ar@{->>}[r]^-{\ov{\alpha}} & 
A_n \ar[r] & A}
\end{equation}
such that the composite map 
$A_n \to Q_n[y_3,y_4] \xrightarrow{\ov{\alpha}} A_n$ on the bottom is identity. 
Since $A_n$ is smooth over $R_n$ so that $D_1({A_n}/{R_n}) = 0$, we have the
Jacobi-Zariski exact sequence
\begin{equation}\label{eqn:local-AQ-0}
0 \to D_1({A_n}/{Q_n}) \to \Omega^1_{{Q_n}/{R_n}} \otimes_{Q_n} A_n 
\xrightarrow{\ov{\alpha}}
\Omega^1_{{A_n}/{R_n}} \to \Omega^1_{{A_n}/{Q_n}} \to 0.
\end{equation}
We can write $Q_n = {R_n[x_2, x_3, x_4]}/{(x_1x_2-x_3x_4, \fm^n_R)}$
and this yields
\begin{equation}\label{eqn:local-AQ-1}
\Omega^1_{{Q_n}/{R_n}} \otimes_{Q_n} A_n \simeq \frac{A_n dx_2 \oplus A_n dx_3 
\oplus A_n dx_4}{(x_1dx_2 - x_3dx_4 -x_4dx_3, d(\fm^n_R))}.
\end{equation}
Since $\Omega^1_{{A_n}/{R_n}} \simeq A_ndy_3 \oplus A_ndy_4$,
we get (in ~\eqref{eqn:local-AQ-0})
\[
\begin{array}{lll}
\ov{\alpha}(f_2dx_2 + f_3dx_3 + f_4dx_4) = 0 & \Leftrightarrow &
f_2\ov{\alpha}(dx_2) + f_3 \ov{\alpha}(dx_3) + f_4\ov{\alpha}(dx_4) = 0 \\
& \Leftrightarrow & f_2x_1(y_3dy_4 + y_4dy_3) + f_3x_1dy_3 + f_4x_1dy_4 = 0 \\
& \Leftrightarrow & (f_2x_1y_4 + f_3x_1) dy_3 + (f_2x_1y_3 + f_4x_1)dy_4 = 0 \\
& \Leftrightarrow & (f_2y_4 + f_3)x_1 = 0 = (f_2y_3 + f_4)x_1 \ \mbox{in} \
\ A_n.
\end{array}
\] 

Representing $f_i$'s in $R[x_1, y_3, y_4]$ (see ~\eqref{eqn:Blow-up-5}),
the last condition above is equivalent to 
$x^n_1 |(f_2y_4 + f_3)x_1$ and $x^n_1 | (f_2y_3 + f_4)x_1$ in $R[x_1, y_3, y_4]$.
Since $R$ is an integral domain, we equivalently get
$f_3  = x^{n-1}_1g_3 - f_2y_4$ and $f_4 = x^{n-1}_1g_4 - f_2y_3$ for some
$g_3, g_4  \in R[x_1, y_3, y_4]$.
This concludes the proof.
\end{proof}

\vskip .3cm

The key result to estimate $H^0(E_n, {_R\wt{\Omega}^1_{E_n}})$ is the following.

\begin{lem}\label{lem:Key-Key}
For every integer $n \ge 2$, the vertical arrows in ~\eqref{eqn:Main-diagram} 
give rise to a commutative diagram of short exact sequences of
Zariski sheaves on $\P^3_k$:
\begin{equation}\label{eqn:Key-Key-0}
\xymatrix@C1pc{
0 \ar[r] & {_R\sD_1({E'_n}/{P'_n})}\otimes_{E'_n} \sO_{E'} \ar[r]^-{\gamma'}
\ar[d] & {_R\sD_1({E'}/{P'_n})} \ar[r]^-{\eta'} \ar[d] & 
{_R\sO_{E'}(1)} \ar[r] \ar[d] & 0 \\
0 \ar[r] & {_R\sD_1({E_n}/{P_n})}\otimes_{E_n} \sO_{E} \ar[r]_-{\gamma}
& {_R\sD_1({E}/{P_n})} \ar[r]_-{\eta} & {_R\sO_{E}(1)} \ar[r] & 0.}
\end{equation}
\end{lem}
\begin{proof}
The vertical arrows are the natural restriction maps
via the closed embedding $E \inj \P^3_k$. We shall prove the
exactness of the bottom row as the argument for the top row is
identical and simpler.

We have the Jacobi-Zariski exact sequence
\[
{_R\sD_1({E_n}/{P_n})}\otimes_{E_n} \sO_{E} \xrightarrow{\gamma}
{_R\sD_1({E}/{P_n})} \xrightarrow{\eta} 
{_R{\sI}/{\sI^2}} \simeq {_R\sO_E(1)} := R \otimes_k \sO_E(1),
\]
where $\sI$ is the sheaf of ideals of the embedding $E \inj V$.
We need to show that $\gamma$ is injective and $\eta$ is surjective.
It suffices to prove these two assertions in our local situation.
We shall prove the surjectivity of $\eta$ first.
We continue to follow the notations of the local coordinate rings
in ~\eqref{eqn:Blow-up-7}.

Since ${I}/{I^2} = {(x_1)}/{(x^2_1)} \simeq D_1({A}/{A_n})$, we need
to show that the map $D_1({A}/{Q_n}) \xrightarrow{\eta} D_1({A}/{A_n})$,
induced by the map $Q_n \inj Q_n[y_3, y_4] \xrightarrow{\ov{\alpha}} A_n$,
is surjective.

Since $R_n \to A_n$ is a smooth map so that
$D_1({A_n}/{R_n}) = 0$, there is a Jacobi-Zariski exact sequence of the
form
\begin{equation}\label{eqn:Key-Key-1}
0 \to D_1({A}/{R_n}) \to 
D_1({A}/{A_n}) \to \Omega^1_{{A_n}/{R_n}} \otimes_{A_n} A \to \Omega^1_{A/{R_n}} 
\to 0.
\end{equation}

Since $\Omega^1_{{R/{R_n}}} = 0$, the fundamental exact sequence of
K{\"a}hler differentials for the maps $R_n \to R \to A$ coming from the
commutative square
\begin{equation}\label{eqn:Key-Key-2}
\xymatrix@C1pc{
R_n \ar[r] \ar[d]_{\phi_n} & R \ar[d] \\
A_n \ar[r] & A_n \otimes_{Q_n} R = A}
\end{equation}
shows that the map $\Omega^1_{{A/{R_n}}} \to \Omega^1_{{A/{R}}}$ is an
isomorphism. On the other hand, the map $\Omega^1_{{A_n}/{R_n}} \otimes_{A_n} A 
\to \Omega^1_{A/{R}}$ is an isomorphism. It follows that the
map $\Omega^1_{{A_n}/{R_n}} \otimes_{A_n} A \to \Omega^1_{A/{R_n}}$ is an
isomorphism too. It follows from ~\eqref{eqn:Key-Key-1} that the map
$D_1({A}/{R_n}) \to D_1({A}/{A_n})$ is an isomorphism.

On the other hand, it follows from ~\eqref{eqn:Key-Key-3} that the square
\begin{equation}\label{eqn:Key-Key-4}
\xymatrix@C1pc{
D_1({A}/{R_n}) \ar[r] \ar[d] & D_1({A}/{Q_n}) \ar[d]^{\eta} \\
D_1({A}/{A_n}) \ar@{=}[r] & D_1({A}/{A_n})}
\end{equation}
commutes, where the left vertical arrow is induced by $\phi_n$
and the right vertical arrow is induced by $\ov{\alpha}$.
Since we have shown above that the left vertical arrow is
an isomorphism, we conclude that $\eta$ is surjective.

To prove the injectivity of $\gamma$, we consider the commutative diagram
\begin{equation}\label{eqn:Key-Key-5}
\xymatrix@C1pc{
D_1({A_n}/{Q_n}) \ar[r]^-{\zeta_n} \ar@{->>}[d]_{\beta_n} & 
\Omega^1_{{Q_n}/{R_n}} \otimes_{Q_n} A_n \ar[d]^{\theta_n} \\
D_1({A_n}/{Q_n}) \otimes_{A_n} A \ar[r]^-{\psi_n} \ar[d]_{\gamma} &
\Omega^1_{{Q_n}/{R_n}} \otimes_{Q_n} A \ar@{=}[d] \\
D_1({A}/{Q_n}) \ar[r] & \Omega^1_{{Q_n}/{R_n}} \otimes_{Q_n} A.}
\end{equation}

Using this diagram, it suffices to show that $\psi_n$ is injective.
Let $w \in D_1({A_n}/{Q_n})$ be such that $\psi_n \circ \beta_n(w) = 0$.
Equivalently, we get $\theta_n \circ \zeta_n(w) = 0$.
Using \lemref{lem:local-AQ}, we can write $w = fdx_2 + (x^{n-1}_1g - fy_4)dx_3
+ (x^{n-1}_1h - fy_3)dx_4$, where $f,g,h \in A_n$.
It follows from ~\eqref{eqn:local-AQ-1} and the definition of
$\ov{\alpha}$ in ~\eqref{eqn:Blow-up-5} that
$\Omega^1_{{Q_n}/{R_n}} \otimes_{Q_n} A \simeq Adx_2 \oplus Adx_3 \oplus Adx_4$.
It follows that $\theta_n \circ \zeta_n(w) = 0$ if and only if
$f = x^{n-1}_1g - fy_4 = x^{n-1}_1h -fy_3 = 0$ modulo $(x_1)$ in $A_n$.
Since $n \ge 2$, this condition is equivalent to $f = 0$ modulo $(x_1)$.

Let us write $f = x_1f'$. We then get
$fdx_2 -fy_4 dx_3 - fy_3 dx_4 = f'(x_1dx_2 - x_1y_4dx_3 - x_1y_3 dx_4)$.
On the other hand, Lemma~\ref{lem:local-AQ} says that the term on the
right hand side of this equality is zero in $D_1({A_n}/{Q_n})$.
Hence, we get
\begin{equation}\label{eqn:Key-Key-6}
\begin{array}{lll}
w & = & fdx_2 + (x^{n-1}_1g - fy_4)dx_3 + (x^{n-1}_1h - fy_3)dx_4 \\
& = &  f'(x_1dx_2 - x_1y_4dx_3 - x_1y_3 dx_4) + (x^{n-1}_1g dx_3 + 
x^{n-1}_1h dx_4) \\
& = & x^{n-1}_1 g dx_3 + x^{n-1}_1 h dx_4.
\end{array}
\end{equation}

We conclude from ~\eqref{eqn:Key-Key-6} that
\begin{equation}\label{eqn:Key-Key-7}
\Ker(\theta_n \circ \zeta_n) = \frac{(x^{n-1}_1)A_n dx_3 \oplus 
(x^{n-1}_1)A_n dx_4}{J_n}.
\end{equation}
Since such elements clearly die via $\beta_n$, we conclude that
$\psi_n$ is injective. The proof of the lemma is now complete.
\end{proof}

\begin{lem}\label{lem:pro-iso-D}
The canonical map $\{\beta_n\}_n: \{{_R\sD_1({E_n}/{P_n})}\}_n \to 
\{{_R\sD_1({E_n}/{P_n})} \otimes_{E_n} \sO_E\}_n$ of sheaves of pro-abelian 
groups on $E$ is an isomorphism.
\end{lem}
\begin{proof}
In the local notations of ~\eqref{eqn:Blow-up-7}, we have seen in the last 
part of the proof of \lemref{lem:Key-Key} (see ~\eqref{eqn:Key-Key-5}) that
$\ker(\beta_n) = \Ker(\psi_n \circ \beta_n) = \Ker(\theta_n \circ \zeta_n)$.
It follows from ~\eqref{eqn:Key-Key-7} that 
the canonical map $\Ker(\theta_{n+1} \circ \zeta_{n+1})
\to \Ker(\theta_{n} \circ \zeta_{n})$ is zero.
We have thus proven that the restriction of the map
$\Ker(\beta_{n+1}) \to \Ker(\beta_n)$ to the affine open
subset $\{y_1 \neq 0\}$ of $\P^3_k$ is zero for every $n \ge 1$.
Since the same argument works for any open subset $\{y_i \neq 0\}$,
we conclude that the map of Zariski sheaves
$\Ker(\beta_{n+1}) \to \Ker(\beta_n)$ is zero. Since each $\beta_n$ is
clearly surjective, we are done.
\end{proof}

\begin{lem}\label{lem:pro-iso-O}
Let $\fm_R$ denote the ideal of $R[x_1, \cdots , x_4]$ defining the origin in
$\A^4_R$. Then, there exists a commutative diagram
\begin{equation}\label{eqn:pro-iso-O-1}
\xymatrix@C1pc{
{\fm_R}/{\fm^2_R} \otimes_k \sO_{\P^3_k} \ar[r]^-{\simeq} \ar@{->>}[d] &
{_R\sD_1({P}/{P'_n})} \otimes_k \sO_{\P^3_k} \ar[r]^-{\simeq} \ar@{->>}[d] &
{_R\sD_1({\P^3_k}/{P'_n})} \ar@{->>}[d] \\
{\fm_R}/{\fm^2_R} \otimes_k \sO_{E} \ar[r]_-{\simeq}  &
{_R\sD_1({P}/{P_n})} \otimes_k \sO_{E} \ar[r]_-{\simeq} &
{_R\sD_1({E}/{P_n})}.}
\end{equation}
\end{lem}
\begin{proof}
The vertical arrows are the restriction maps induced by the closed
embedding $E \inj \P^3_k$. We prove that the horizontal arrows on the bottom
are isomorphisms and the same proof applies to the arrows on the top.

The Jacobi-Zariski exact sequence for the maps $E \to P = \Spec(k) \inj P_n$
and smoothness of $E$ over $k$ yield isomorphisms
${\fm_R}/{\fm^2_R} \otimes_k \sO_E \xrightarrow{\simeq} 
{_R\sD_1({P}/{P_n})} \otimes_k \sO_E \xrightarrow{\simeq} {_R\sD_1({E}/{P_n})}$.
The commutativity of the squares follows from the naturality of these
isomorphisms.
\end{proof}

\begin{prop}\label{prop:Pro-iso-vb}
For any integer $n \ge 2$, there is a natural isomorphism
of Zariski sheaves on $E$:
\[
{_R\sD_1({E_n}/{P_n})} \otimes_{E_n} \sO_E \xrightarrow{\simeq}
{_R\Omega^1_{{\P^3_k}/k}|_E (1)}.
\]
In particular, there is an isomorphism of pro-sheaves of $\sO_E$-modules
$\{{_R\sD_1({E_n}/{P_n})}\}_n \xrightarrow{\simeq}  
{_R\Omega^1_{{\P^3_k}/k}|_E (1)}$.
\end{prop}
\begin{proof}
In view of \lemref{lem:pro-iso-D}, we only have to prove the first isomorphism.
If we let $R[x_1, \cdots , x_4] = B = \oplus_{i \ge 0} B_i$, then recall that
the canonical map $\nu:B[-1]^{4} \to B$, given by $e_i \mapsto x_i$,
defines a surjective map of sheaves $\sO_{\P^3_k}(-1)^{4} \surj \sO_{\P^3_k}$ 
whose kernel is $\Omega^1_{{\P^3_k}/k}$. Twisting by $\sO_{\P^3_k}(1)$ and 
restricting to $E$, we get the exact Euler sequence
\begin{equation}\label{eqn:Pro-iso-vb-0}
0 \to \Omega^1_{{\P^3_k}/k}|_E(1) \to \sO_{E}^4 \xrightarrow{\nu} 
\sO_{E}(1) \to 0.
\end{equation}
Using the canonical isomorphism ${\fm}/{\fm^2} \xrightarrow{\simeq}
B_1 \xrightarrow{\simeq} A_1 \otimes_k A'_1$ (see notation below
~\eqref{eqn:H0-Omega-1}), where the latter map is an isomorphism
induced by $\phi_k$ in ~\eqref{eqn:GMA-0}, we can also write 
the middle term of the above exact sequence as
${\fm}/{\fm^2} \otimes_k \sO_E$.
Combining this with \lemref{lem:pro-iso-O}, the exact sequence
can be written in the form
\[
0 \to \Omega^1_{{\P^3_k}/k}|_E(1) \to \sD_1({E}/{P_n}) \xrightarrow{\nu} 
\sO_{E}(1) \to 0
\]
and one checks from the proof of \lemref{lem:Key-Key} that the map
$\nu$ as defined above coincides with the map $\eta$ in ~\eqref{eqn:Key-Key-0}.
Hence, we conclude the proof by tensoring the exact sequence with $R$ over $k$ 
and using \lemref{lem:Key-Key}.
\end{proof}

\begin{cor}\label{cor:Pro-vanish-D}
The pro-abelian group $\{H^i(E_n, {_R\sD_1({E_n}/{P_n})})\}_n$ is zero
for $i = 0,1$.
\end{cor}
\begin{proof}
By \propref{prop:Pro-iso-vb}, it suffices to show that
$H^i(E, {_R\Omega^1_{{\P^3_k}/k}|_E(1)}) \simeq H^i(E, \Omega^1_{{\P^3_k}/k}|_E(1)) 
\otimes_k$
\nolinebreak
$R$ is zero for $i = 0,1$. Using ~\eqref{eqn:Pro-iso-vb-0}, we need to show that
$H^0(E, \sO^4_E) \xrightarrow{\nu} H^0(E, \sO_E(1))$ is an isomorphism and
$H^1(E, \sO_E) = 0$. But both follow from \lemref{lem:Coh-main}.
\end{proof}

\begin{thm}\label{thm:Surj-Main}
The blow-up morphism $\pi: V \to X$ in ~\eqref{eqn:Main-diagram}
induces a surjective map of pro-abelian groups
\[
\{{_R\wt{HC}_1(P_n)}\}_n \stackrel{\pi^*}{\surj}
\{H^0(E_n, {_R\wt{\sHC}_{1,E_n}})\}_n.
\]
\end{thm}
\begin{proof}
In view of \lemref{lem:surjection-H0}, it suffices to show that the map
of pro-abelian groups $\{{_R\wt{\Omega}^1_{P_n}}\}_n \to 
H^0(E_n, {_R\wt{\Omega}^1_{E_n}})\}_n$ is surjective.
We shall in fact show that this map is an isomorphism.

For every $n \ge 1$, we have the
Jacobi-Zariski exact sequence of Zariski sheaves on $E$:
\begin{equation}\label{eqn:Surj-Main-2}
{_R\sD_1({E_n}/{P})} \to {_R\sD_1({E_n}/{P_n})} \to {_R\Omega^1_{{P_n}/{P}}} 
\otimes_{{P_n}} \sO_{E_n} \to {_R\Omega^1_{{E_n}/{P}}} \to 
{_R\Omega^1_{{E_n}/{P_n}}}
\to 0.
\end{equation}
It follows from \lemref{lem:AQ-1} that $\{{_R\sD_1({E_n}/{P})}\}_n = 0$.
We thus get a commutative diagram of exact sequences of sheaves of 
pro-abelian groups
\begin{equation}\label{eqn:Surj-Main-3}
\xymatrix@C1pc{
0 \ar[r] & \{{_R\sD_1({E_n}/{P_n})}\}_n \ar[r] &
\{{_R\Omega^1_{{P_n}/{P}}} \otimes_{{P_n}} \sO_{E_n}\}_n \ar[r] &
\{{_R\Omega^1_{{E_n}/{P}}}\}_n \ar[r] \ar[d] & 
\{{_R\Omega^1_{{E_n}/{P_n}}}\}_n \ar[r] \ar[d] & 0 \\
& & & {_R\Omega^1_{{E}/{P}}} \ar@{=}[r] & {_R\Omega^1_{{E}/{P}}}.}
\end{equation}

Taking the kernels and using \lemref{lem:H0-Omega}, we get an exact sequence
\begin{equation}\label{eqn:Surj-Main-4}
0 \to \{{_R\sD_1({E_n}/{P_n})}\}_n \to \{{_R\Omega^1_{{P_n}/{P}}} 
\otimes_{{P_n}} \sO_{E_n}\}_n \to \{{_R\wt{\Omega}^1_{{E_n}/{P}}}\}_n \to 0.
\end{equation}

Considering the cohomology and using \corref{cor:Pro-vanish-D},
we get $\{H^0(P_n, {_R\Omega^1_{{P_n}/{P}}}\otimes_{{P_n}} \sO_{E_n})\}_n 
\xrightarrow{\simeq} \{H^0(E_n, {_R\wt{\Omega}^1_{{E_n}/{P}}}\}_n$.
On the other hand, we have
$H^0(P_n, {_R\Omega^1_{{P_n}/{P}}}\otimes_{{P_n}} \sO_{E_n})
\simeq {_R\Omega^1_{{Q_n}/k}} \otimes_{Q_n} H^0(E_n, \sO_{E_n})$
(recall here that $P_n = \Spec(Q_n)$).
Since $E/k$ is geometrically integral and is the only fiber of the
map $E_n \to P_n$, it follows that the canonical map $Q_n \to
H^0(E_n, \sO_{E_n})$ is an isomorphism (see, for instance,
\cite[Chap.~5, Exc.~3.12]{Liu}).
We conclude that 
$H^0(P_n, {_R\Omega^1_{{P_n}/{P}}}\otimes_{{P_n}} \sO_{E_n})
\simeq {_R\Omega^1_{{Q_n}/k}}$. In particular, the canonical map
\begin{equation}\label{eqn:Surj-Main-5}
\{{_R\Omega^1_{{Q_n}/k}}\}_n \to \{H^0(E_n, {_R\wt{\Omega}^1_{{E_n}/{P}}})\}_n
\end{equation}
is an isomorphism.

We now consider the commutative diagram of short exact sequences
\begin{equation}\label{eqn:Surj-Main-6}
\xymatrix@C1pc{
0 \ar[r] & \Omega^1_R \otimes_k \sO_{E_n} \ar[r] \ar[d] &
{_R\Omega^1_{E_n}} \ar[r] \ar[d] & {_R\Omega^1_{{E_n}/{P}}} \ar[r] \ar[d] & 0 \\
0 \ar[r] & \Omega^1_R \otimes_k \sO_{E} \ar[r] &
{_R\Omega^1_{E}} \ar[r] & {_R\Omega^1_{{E}/{P}}} \ar[r] & 0,}
\end{equation}
where the rows are exact from the left because they are locally split.
Taking the kernels, we get a short exact sequence
\begin{equation}\label{eqn:Surj-Main-7}
0 \to \Omega^1_R \otimes_k {\sI}/{\sI^n} \to
{_R\wt{\Omega}^1_{E_n}} \to {_R\wt{\Omega}^1_{{E_n}/{P}}} \to 0.
\end{equation}
Using the filtration ~\eqref{eqn:extra}, \lemref{lem:Coh-main} and induction
on $n \ge 1$, we see that $H^1(E_n, \Omega^1_R \otimes_k {\sI}/{\sI^n})
\simeq \Omega^1_R \otimes_k H^1(E_n, {\sI}/{\sI^n}) = 0$
and hence there is a short exact sequence
\begin{equation}\label{eqn:Surj-Main-8}
0 \to \Omega^1_R \otimes_k H^0(E_n, {\sI}/{\sI^n}) \to
H^0(E_n, {_R\wt{\Omega}^1_{E_n}}) \to H^0(E_n, {_R\wt{\Omega}^1_{{E_n}/{P}}}) \to 0.
\end{equation}

Comparing this with a similar exact sequence
for $P_n$, we get a commutative diagram of short exact sequences of
pro-abelian groups

\begin{equation}\label{eqn:Surj-Main-9}
\xymatrix@C1pc{
0 \ar[r] & \{\Omega^1_R \otimes_k {\fm_R}/{\fm^n_R}\}_n \ar[r] \ar[d] &
{\{_R\wt{\Omega}^1_{P_n}}\}_n \ar[r] \ar[d] & \{{_R\wt{\Omega}^1_{{P_n}/{P}}}\}_n
\ar[d] \ar[r] & 0 \\
0 \ar[r] & \{\Omega^1_R \otimes_k H^0(E_n, {\sI}/{\sI^n})\}_n \ar[r] &
\{H^0(E_n, {_R\wt{\Omega}^1_{E_n}})\}_n \ar[r] & 
\{H^0(E_n, {_R\wt{\Omega}^1_{{E_n}/{P}}})\}_n
\ar[r] & 0.}
\end{equation}
 
We have shown above that the map $Q_n \to H^0(E_n, \sO_{E_n})$ is an isomorphism
for every $n \ge 1$. Using the compatible splittings
$Q_n \simeq k \oplus {\fm}/{\fm^n}$ and $H^0(E_n, \sO_{E_n}) \simeq
H^0(E, \sO_E) \oplus H^0(E_n, {\sI}/{\sI^n})$, one checks easily that
${\fm}/{\fm^n} \to H^0(E_n, {\sI}/{\sI^n})$ is an isomorphism.
In particular, the left vertical arrow in ~\eqref{eqn:Surj-Main-9}
is an isomorphism. The right vertical arrow is an isomorphism
by ~\eqref{eqn:Surj-Main-5}. We conclude that the middle vertical 
arrow is also an isomorphism.
This completes the proof of the theorem.
\end{proof} 

\section{Generalization of \thmref{thm:Surj-Main}}\label{sec:Alg}
We do not know if \thmref{thm:Surj-Main} is valid for higher cyclic
homology groups in general. However, we shall show in this section that
this is indeed the case for $i \neq 2$ when the ground field $k$ is algebraic 
over $\Q$. This will be obtained as a special case of the 
more general result that we shall prove when $k$ is any field of 
characteristic zero.

Let $\wt{p}: V \to E$ be the line bundle map
of ~\eqref{eqn:Main-diagram}.
Using the fact that $\Omega^1_{V/E} \simeq \wt{p}^*(\sO_E(1))$, the 
fundamental sequence for K{\"a}hler differentials gives rise to
exact sequence of Zariski sheaves
\[
0 \to \wt{p}^*(\Omega^i_{E/k}) \to \Omega^i_{V/k} \to 
\wt{p}^*(\Omega^{i-1}_{E/k}(1)) \to 0
\]
for every $i \ge 1$. Since this is an exact sequence of vector bundles,
it remains exact on restriction to $E_n$. Taking this restriction and
subsequent push-forward via the map $p:E_n \to E$, we get a
short exact sequence of Zariski sheaves on $E$:
\begin{equation}\label{eqn:Alg-0}
0 \to \Omega^i_{E/k}\otimes_{E} p_*(\sO_{E_n}) \to p_*(\Omega^i_{V/k}|_{E_n}) \to
\Omega^{i-1}_{E/k}(1) \otimes_{E} p_*(\sO_{E_n}) \to 0.
\end{equation}

Using the exact sequence
\[
0 \to \Omega^i_{E/k}(n-1) \to \Omega^i_{E/k}\otimes_{E} 
p_*(\sO_{E_n}) \xrightarrow{\alpha} 
\Omega^i_{E/k}\otimes_{E} p_*(\sO_{E_{n-1}}) \to 0
\]
and an induction on $n \ge 1$, it follows from \lemref{lem:Coh-main} that
$H^0(E, \Omega^i_{E/k}\otimes_{E} p_*(\sO_{E_n})) \simeq$ \\
$\stackrel{n-1}{\underset{j =0}\oplus} H^0(E, \Omega^i_{E/k}(j))$.
Moreover, we have a short exact sequence
\begin{equation}\label{eqn:Alg-1}
0 \to \stackrel{n-1}{\underset{j =0}\oplus} H^0(E, \Omega^i_{E/k}(j)) \to
H^0(E_n, \Omega^i_{V/k}|_{E_n}) \xrightarrow{\alpha} 
\stackrel{n}{\underset{j =1}\oplus} H^0(E, \Omega^{i-1}_{E/k}(j)) \to 0.
\end{equation}


It is easy to check that for every $i \ge 1$, the map
$d: \wt{p}_*(\Omega^{i-1}_{V/k}) \to \wt{p}_*(\Omega^{i}_{V/k})$
induces $d: p_*(\Omega^{i-1}_{V/k}|_{E_{n+1}})
\to p_*(\Omega^{i}_{V/k}|_{E_n})$.
In particular, there is a (strict) map of sheaves of pro-abelian groups
\begin{equation}\label{eqn:ext-d}
d: \{p_*(\Omega^{i-1}_{V/k}|_{E_n})\}_n \to \{p_*(\Omega^{i}_{V/k}|_{E_n})\}_n
\end{equation}
on $E$.

We now claim that the composite map (see ~\eqref{eqn:Alg-0})
\begin{equation}\label{eqn:Alg-1**}
\Omega^{i-1}_{E/k}\otimes_{E} p_*(\sO_{E_{n+1}}) \to p_*(\Omega^{i-1}_{V/k}|_{E_{n+1}})
\xrightarrow{d} p_*(\Omega^{i}_{V/k}|_{E_n}) \to
\Omega^{i-1}_{E/k}(1) \otimes_{E} p_*(\sO_{E_n})
\end{equation}
is surjective with kernel $\Omega^{i-1}_{E/k}$ (via the inclusion
$p^*: \sO_E \inj p_*(\sO_{E_{n+1}})$). In particular,
the composite map
$H^0(E, \Omega^{i-1}_{E/k}(j)) \to H^0(E,  p_*(\Omega^{i-1}_{V/k}|_{E_{n+1}}))
\xrightarrow{d}  H^0(E,  p_*(\Omega^{i}_{V/k}|_{E_{n}})) \to
H^0(E, \Omega^i_{E/k}(j))$ is an isomorphism for every $1 \le j \le n$.

Since this claim is a local assertion, we can use our notations of
~\eqref{eqn:Blow-up-7} and note that $\alpha$ is induced by 
the map $\Omega^i_{V/k} \to \Omega^i_{V/E}$. In the local 
notations, the composite map of ~\eqref{eqn:Alg-1**} is then given by
$\Omega^{i-1}_{A/k} \otimes_{A} {A_{n+1}} \to 
\Omega^{i-1}_{A/k} \otimes_{A} \Omega^1_{{A[x_1]}/A} \otimes_{A[x_1]} A_n 
\simeq \Omega^{i-1}_{A/k} \otimes_{A} A_ndx_1$, that 
sends $w \otimes x^j_1$ to $w \otimes jx^{j-1}_1dx_1$.
Here, we use $A_n = {A[x_1]}/{(x^n_1)}$. This map clearly satisfies the
assertion of the claim.

It follows from the claim that there is a commutative diagram 
\begin{equation}\label{eqn:Alg-1**-0}
\xymatrix@C1pc{
& & \{H^0(E, {p}_*(\Omega^{i-1}_{V/k}|_{E_n}))\}_n \ar[d]^{d} \ar[dr] & & \\
0 \ar[r] & \{\stackrel{n-1}{\underset{j =0}\oplus} H^0(E, \Omega^i_{E/k}(j))\}_n 
\ar[r] & \{H^0(E_n, \Omega^i_{V/k}|_{E_n})\}_n \ar[r]^-{\alpha} &
\{\stackrel{n}{\underset{j =1}\oplus} H^0(E, \Omega^{i-1}_{E/k}(j))\}_n \ar[r] 
& 0}
\end{equation}
of pro-abelian groups such that the diagonal arrow on the right is surjective.
In particular, we conclude from \lemref{lem:Pro-ker-Co} that 
map of pro-abelian groups
\begin{equation}\label{eqn:Alg-1**-1}
\left\{\stackrel{n-1}{\underset{j =0}\oplus} H^0(E, \Omega^i_{E/k}(j))\right\}_n
\to
\left\{\frac{H^0(E, p_*(\Omega^i_{V/k}|_{E_n}))}
{d(H^0(E, p_*(\Omega^{i-1}_{V/k}|_{E_n})))}\right\}_n
\end{equation}
is surjective.

On the other hand, as the sheaf of pro-abelian groups
$\{{\sI^n}/{\sI^{2n}}\}_n$ is zero, the exact sequence
\begin{equation}\label{eqn:Alg-2**}
{\sI^n}/{\sI^{2n}} \xrightarrow{d} \Omega^1_{V/k}|_{E_n} \to
\Omega^1_{{E_n}/k} \to 0
\end{equation}
shows that
$\{p_*(\Omega^1_{V/k}|_{E_n})\}_n \xrightarrow{\simeq} 
\{p_*(\Omega^1_{{E_n}/k})\}_n$.
In particular, \lemref{lem:ext-pro} implies that 
there is a pro-isomorphism of the sheaves of differential graded algebras 
$\{p_*(\Omega^*_{V/k}|_{E_n})\}_n \xrightarrow{\simeq} 
\{p_*(\Omega^*_{{E_n}/k})\}_n$. 
Combining this with ~\eqref{eqn:Alg-1**-1}, we
get for $i \ge 1$, a surjective map of pro-abelian groups
\begin{equation}\label{eqn:Alg-3}
\left\{\stackrel{n-1}{\underset{j =0}\oplus} H^0(E, \Omega^i_{E/k}(j))\right\}_n
\surj 
\left\{\frac{H^0(E_n, \Omega^i_{{E_n}/k})}{d(H^0(E_n, \Omega^{i-1}_{{E_n}/k}))}
\right\}_n.
\end{equation}

Using this surjective map, 
we can prove:

\begin{prop}\label{prop:Alg-Main}
Let $k$ be a field of characteristic zero.
Then the map of pro-abelian groups
$\{\wt{HC}^{k}_i(P_n)\}_n \xrightarrow{\pi^*}
\{H^0(E_n, \wt{\sHC}^{k}_{i,E_n})\}_n$ is surjective for $i \ge 3$.
\end{prop}
\begin{proof}
Using \lemref{lem:Pro-trivial} and the identification $P_n = \Spec(Q_n)$, 
we can replace the right hand side by
$\{H^0(E_n, \wt{\sHC}^{k,(i)}_{i,E_n})\}_n$
and the left hand side by $\{\wt{HC}^{k,(i)}_i(Q_n)\}_n$. 
Using the isomorphisms $\wt{\Omega}^i_{{Q_n}/k}
\simeq \Omega^i_{{Q_n}/k}$, we get a commutative diagram
\begin{equation}\label{eqn:Alg-Main-0}
\xymatrix@C1pc{
\wt{HC}^{k,(i)}_i(Q_n) \ar[d] &
\frac{\wt{\Omega}^i_{{Q_n}/k}}{d(\wt{\Omega}^{i-1}_{{Q_n}/k})}
\ar[r]^{\simeq} \ar[d] \ar[l]_{\simeq} &
\frac{\Omega^i_{{Q_n}/k}}{d(\Omega^{i-1}_{{Q_n}/k})} \ar[d]  \\
\{H^0(E_n, \wt{\sHC}^{k,(i)}_{i,E_n})\}_n &
\frac{H^0(E_n, \wt{\Omega}^{i}_{{E_n}/k})}
{d(H^0(E_n, \wt{\Omega}^{i-1}_{{E_n}/k}))} \ar[r]^{\simeq} \ar@{->>}[l] &
\frac{H^0(E_n, {\Omega}^{i}_{{E_n}/k})}
{d(H^0(E_n, {\Omega}^{i-1}_{{E_n}/k}))}.}
\end{equation}

Since $\Omega^{\ge 3}_{E/k} = 0$, it follows from the splitting of the map 
$p:E_n  \to E$ via the section $E \inj E_n$ that the horizontal arrow on the 
bottom right is an isomorphism. The horizontal arrow on the 
bottom left is surjective by \corref{cor:surjection-H0-ext}.
Hence, it suffices to show that the map
$\left\{\frac{\Omega^i_{{Q_n}/k}}{d(\Omega^{i-1}_{{Q_n}/k})}\right\}_n \to
\left\{\frac{H^0(E_n, {\Omega}^{i}_{{E_n}/k})}
{d(H^0(E_n, {\Omega}^{i-1}_{{E_n}/k}))}\right\}_n$
is surjective for $i \ge 3$. But the right hand side of this map 
is in fact zero by ~\eqref{eqn:Alg-3} because $\Omega^{\ge 3}_{E/k} = 0$.
This finishes the proof.
\end{proof}

\begin{remk}\label{remk:i=2}
One can check from Lemmas~\ref{lem:Coh-main} and ~\ref{lem:surjection-H0}
that the horizontal arrow on the 
bottom right of ~\eqref{eqn:Alg-Main-0} is an isomorphism and
the one on the left is surjective for $i =2$ as well.
However, we do not know how to conclude using ~\eqref{eqn:Alg-3} that 
$\left\{\frac{\Omega^2_{{Q_n}/k}}{d(\Omega^{1}_{{Q_n}/k})}\right\}_n \to
\left\{\frac{H^0(E_n, {\Omega}^{2}_{{E_n}/k})}
{d(H^0(E_n, {\Omega}^{1}_{{E_n}/k}))}\right\}_n$ is surjective.
\end{remk}

\section{The main theorems}\label{sec:MT}
To prove our main results of this text, we shall use the following
straightforward variant of the Thomason-Trobaugh spectral sequence
for algebraic $K$-theory. We shall continue to use the
notations of \S~\ref{sec:R-reg}.

\begin{thm}$($\cite[Theorem~10.3]{TT}$)$\label{thm:TT-gen}
Let $k$ be a field and let $R$ be a commutative Noetherian $k$-algebra.
Given $X \in \Sch_k$ and a closed subscheme $Z \subset X$, there exists a 
strongly convergent spectral sequence
\begin{equation}\label{eqn:TT-S}
E^{p,q}_2 = H^p(X, {_R\sK_{q,(X, Z)}}) \Rightarrow {_RK_{q-p}(X,Z)}.
\end{equation}
\end{thm}
\begin{proof}
This is proved by repeating the argument of \cite[Theorem~10.3]{TT}
verbatim with the aid of \cite[Proposition~3.20.2]{TT} and the fact that
if $U = \Spec(A)$ is an affine open in $X$ with a prime ideal $\fp \subset A$
such that $A_{\fp} = {\underset{i}\varinjlim} \ A[f^{-1}_i]$,
then $A_{\fp} \otimes_k R \simeq {\underset{i}\varinjlim} \ 
(A\otimes_k R)[f^{-1}_i]$. We leave out the details.
\end{proof} 

\subsection{$K$-theory of $R[M]$ and proof of \thmref{thm:Main-1}}
Let $k$ be any field of characteristic zero and let $R$ be 
a Noetherian regular $k$-algebra. Let $R[M]$ be Gubeladze's monoid
algebra of ~\eqref{eqn:GMA-0}. Let $X = \Spec(k[M])$.

It follows from \cite[Theorem~0.1]{Morrow} that there is a long
exact sequence
\begin{equation}\label{eqn:MT-1-0}
\cdots \to {_RK_{m+1}(V, E)} \oplus \{{_RK_{m+1}(P_n, P)}\}_n \to
\{{_RK_{m+1}(E_n, E)}\}_n \to {_RK_m(X, P)} \hspace*{3cm}
\end{equation}
\[ 
\hspace*{5cm} 
\to 
{_RK_{m}(V, E)} \oplus \{{_RK_{m}(P_n,P)}\}_n \to \{{_RK_m(E_n, E)}\}_n 
\to \cdots .
\]

Note that the exact sequence ~\eqref{eqn:MT-1-0}
already follows from the pro-descent theorem \cite[Theorem~1.1]{Krishna-2}
and one does not need to use the more general result of \cite{Morrow}.

Using the homotopy invariance for the map $\wt{p}: V \to E$, this
exact sequence becomes (see \S~\ref{sec:Not-R} for the definition
of ${_R\wt{K}_*(-)}$)
\begin{equation}\label{eqn:MT-1-1}
\cdots \to 
\{{_R\wt{K}_{m+1}(P_n)}\}_n \to
\{{_R\wt{K}_{m+1}(E_n)}\}_n \to {_R\wt{K}_m(X)} 
\to \{{_R\wt{K}_{m}(P_n)}\}_n \to \{{_R\wt{K}_m(E_n)}\}_n 
\to \cdots .
\end{equation}

\begin{lem}\label{lem:MT-2}
There is an isomorphism of pro-abelian groups
\[
{\rm Coker}(\{{_R\wt{K}_{m+1}(P_n)}\}_n \to
\{{_R\wt{K}_{m+1}(E_n)}\}_n) \hspace*{5cm}
\]
\[
\hspace*{5cm}
\simeq
{\rm Coker}(\{{_R\wt{HC}^{(m)}_m(P_n)}\}_n \to \{H^0(E_n,
{_R\wt{\sHC}^{(m)}_{m,E_n}})\}_n).
\]
\end{lem}
\begin{proof}
It follows from the spectral sequence ~\eqref{eqn:TT-S},
the isomorphism ${_R\wt{\sK}_{i,E_n}} \simeq {_R\wt{\sHC}_{i-1,E_n}}$
and \lemref{lem:vanish-cyclic-1} that
the edge map 
$\{{_R\wt{K}_{m+1}(E_n)}\}_n \to \{H^0(E_n, {_R\wt{\sHC}_{m,E_n}})\}_n$
is an isomorphism.
It follows from \lemref{lem:Pro-trivial} that 
\begin{equation}\label{eqn:MT-2-0}
\left\{{_R\wt{K}_{m+1}(E_n)}\right\}_n \xrightarrow{\simeq}
\left\{H^0(E_n, {_R\wt{\sHC}^{(m)}_{m,E_n}})\right\}_n.
\end{equation}

We conclude the proof of the lemma by combining this with the
isomorphism ${_R\wt{K}_{m+1}(P_n)}$ \\
$\simeq {_R\wt{HC}^{(m)}_{m}(P_n)}$
(using \lemref{lem:Pro-trivial} again).
\end{proof}

Combining ~\eqref{eqn:MT-1-1}, \lemref{lem:MT-2} and
~\eqref{eqn:MT-2-0}, we obtain the following general 
result.

\begin{thm}\label{thm:MT-General}
For any integer $m \in \Z$, there is a short exact sequence
\begin{equation}\label{eqn:MT-General-0}
0 \to \frac{\{H^0(E_n, {_R\wt{\sHC}^{(m)}_{m,E_n}})\}_n}
{\{{_R\wt{HC}^{(m)}_m(P_n)}\}_n} \to {_R\wt{K}_{m}(X)} \hspace*{7cm}
\end{equation}
\[
\hspace*{4cm}
\to
\Ker\left(\{{_R\wt{HC}^{(m-1)}_{m-1}(P_n)}\}_n \to 
\{H^0(E_n, {_R\wt{\sHC}^{(m-1)}_{m-1,E_n}})\}_n\right) \to 0.
\]
\end{thm}

\vskip .4cm

{\sl {Proof of \thmref{thm:Main-1}:}}
By definition of ${_RK_*(X)}$, the theorem is equivalent to the 
assertion that ${_R\wt{K}_1(X)} = 0$. 
The term on the left of the exact sequence ~\eqref{eqn:MT-General-0} is
zero for $m =1$ by \thmref{thm:Surj-Main}.
The term on the right is same as
$\{\Ker({\fm_R}/{\fm^n_R} \to H^0(E_n, {_R{\sI}/{\sI^n}}))\}_n$.
But we have shown in the last part of the proof of \thmref{thm:Surj-Main}
that this map is an isomorphism. We conclude by \thmref{thm:MT-General}.
$\hfill\square$

\begin{cor}\label{cor:MT**}
Let $k$ be a field of characteristic zero and let $R$ 
be a Noetherian regular $k$-algebra.
Then Gubeladze's monoid algebra $R[M]$ is $K_1$-regular.
\end{cor}
\begin{proof}
The proof is already explained after the statement of
\thmref{thm:Main-1} in \S~\ref{sec:Intro}.
\end{proof}

For any augmented $k$-algebra $A$, let $\wt{K}(A) := 
{\rm hofiber}(K(A) \xrightarrow{\epsilon_*} K(k))$ denote the reduced 
$K$-theory spectrum of $A$, where $\epsilon: A \to k$ is the augmentation map.

\begin{thm}\label{thm:Main-alg-K}
Let $k$ be a field which is algebraic over $\Q$ and let
$k[M]$ denote Gubeladze's monoid algebra.
Then the following hold.
\begin{enumerate}
\item
$\wt{K}_i(k[M]) = 0$ for $i \notin \{2,3,4\}$.
\item
$\wt{K}_4(k[M]) \simeq
\left\{\frac{\Omega^{3}_{{Q_n}/k}}{d(\Omega^{2}_{{Q_n}/k})}\right\}_n$.
\item
$\wt{K}_3(k[M]) \simeq 
\left\{\Ker(\frac{\Omega^{2}_{{Q_n}/k}}{d(\Omega^{1}_{{Q_n}/k})}
\to H^0(E_n, \wt{\sHC}^{(2)}_{2,E_n}))\right\}_n$.
\end{enumerate}
\end{thm}
\begin{proof}

The parts (1) and (3) follow by combining Theorems~\ref{thm:Main-1},
~\ref{thm:MT-General} and \propref{prop:Alg-Main} with the following.
\begin{enumerate}
\item
The map $d: \Omega^3_{{Q_n}/k} \to \Omega^4_{{Q_n}/k}$ is surjective
so that $\wt{HC}^{(4)}_{4}(Q_n) = 0$.
\item
$\Omega^{\ge 5}_{{Q_n}/k} = 0$. 
\end{enumerate}

To prove (2), we only need to show that $\{H^0(E_n, \wt{\sHC}^{k,(3)}_{3,E_n})\}_n
= 0$. We shall in fact show this even if $k$ is not algebraic over $\Q$.

By \corref{cor:surjection-H0-ext}, it suffices to show that
$\left\{\frac{H^0(E_n, \wt{\Omega}^3_{{E_n}/k})}
{d(H^0(E_n, \wt{\Omega}^{2}_{{E_n}/k}))}\right\}_n = 0$.
Since the map $\left\{\frac{H^0(E_n, \wt{\Omega}^3_{{E_n}/k})}
{d(H^0(E_n, \wt{\Omega}^{2}_{{E_n}/k}))}\right\}_n
\to 
\left\{\frac{H^0(E_n, \Omega^3_{{E_n}/k})}
{d(H^0(E_n, \Omega^{2}_{{E_n}/k}))}\right\}_n$ is a split inclusion,
it suffices to show that the latter term is zero.
But this follows from ~\eqref{eqn:Alg-3} since $\Omega^3_{E/k} = 0$.
\end{proof}


\subsection{Proof of \thmref{thm:Main-4}}
In order to show that $\wt{K}(k[M])$ is not contractible, we need to
prove the following intermediate step.
Let $B = k[x_1, \cdots ,x_4], Q = {B}/{(x_1x_2-x_3x_4)}, \
Q_n = {Q}/{\bar{\fm}^n}$ so that $P_n = \Spec(Q_n)$ 
(see ~\eqref{eqn:Blow-up-7}).

\begin{lem}\label{lem:No-reg-step}
Let $k$ be any field of characteristic zero.
The map $\Omega^4_{B/k} \surj \Omega^4_{Q/k}$ induces an
isomorphism $k \omega \xrightarrow{\simeq} \Omega^4_{Q/k}$,
where $\omega = dx_1dx_2dx_3dx_4$.
\end{lem}
\begin{proof}
We have $\Omega^{4}_{{B}/k} \simeq B \omega$.
We first show that $x_i \omega = 0$ for all $1 \le i \le 4$ in
$\Omega^4_{{Q}/k}$.

Observe that $(x_1x_2-x_3x_4)dx_1dx_2dx_3=0$
in $\Omega^3_{Q/k}$, hence $d(x_1x_2-x_3x_4)dx_1dx_2dx_3=0$
in $\Omega^4_{Q/k}$. But this implies that
$x_3 dx_1dx_2dx_3dx_4=0$ in $\Omega^4_{Q/k}$.
Similarly, considering zero elements
\[
(x_1x_2-x_3x_4)dx_1dx_2dx_4, \ (x_1x_2-x_3x_4)dx_2dx_3dx_4, \
(x_1x_2-x_3x_4)dx_1dx_3dx_4
\]
in $\Omega^3_{Q/k}$, we get
$x_4 \omega =0 = x_2 \omega = x_1 \omega$ as well in $\Omega^4_{Q/k}$.
It follows that the canonical map $\Omega^4_{B/k} \surj \Omega^4_{Q/k}$
factors through $k\omega \xrightarrow{\simeq} \Omega^4_{B/k} \otimes_B
{B}/{\fm} \surj \Omega^4_{Q/k}$. To prove the claim, it is now enough to
show that $\Omega^4_{Q/k} \neq 0$.

Letting $I = (f = x_1x_2-x_3x_4) \subset B$, we have an exact sequence
of $Q$-modules
\[
{I}/{I^2} \xrightarrow{d} \Omega^1_{B/k} \otimes_B Q \to \Omega^1_{Q/k} \to 0.
\]

It follows from \lemref{lem:Elm-alg-0} that there is a surjection
$\Omega^4_{Q/k} \surj \frac{\Omega^4_{B/k} \otimes_B Q}{F^1}$,
where $F^1$ is the submodule of $\Omega^4_{B/k} \otimes_B Q$ generated by
the exterior products of the form
$\{a_1 \wedge \cdots \wedge a_4| \ a_i \in d(I) \ \mbox{for  \ some} \
1 \le i \le 4\}$. Up to a sign, we can in fact assume that $a_1 \in d(I)$.
It is now immediate that any element
of the form $d(af) \wedge a_2\wedge a_3\wedge a_4$ 
(with $a \in B$) must belong to $\fm \omega$. In particular,
$\omega \neq 0$ in $\frac{\Omega^4_{B/k} \otimes_B Q}{F^1}$
and this shows that $\Omega^4_{Q/k} \neq 0$.
\end{proof}

\begin{thm}\label{thm:No-regular}
Let $k$ be any field of characteristic zero and let $k[M]$ denote 
Gubeladze's monoid algebra. Then the map
$K_4(k) \to K_4(k[M])$ is not an isomorphism. In particular,
$\wt{K}(k[M])$ is not contractible.
\end{thm}
\begin{proof}
Using ~\thmref{thm:MT-General}, it suffices to show that
$\{\Ker(\frac{\wt{\Omega}^{3}_{{Q_n}}}{d(\wt{\Omega}^{2}_{{Q_n}})}
\to H^0(E_n, \wt{\sHC}^{(3)}_{3,E_n}))\}_n \neq 0$.
By the decomposition ~\eqref{eqn:vanish-cyclic-2-3},
there is a split surjection
\[
\Ker\left(\frac{\wt{\Omega}^{3}_{{Q_n}}}{d(\wt{\Omega}^{2}_{{Q_n}})}
\to H^0(E_n, \wt{\sHC}^{(3)}_{3,E_n})\right) \surj
\Ker\left(\frac{\wt{\Omega}^{3}_{{Q_n}/k}}{d(\wt{\Omega}^{2}_{{Q_n}/k})}
\to H^0(E_n, \wt{\sHC}^{k,(3)}_{3,E_n})\right)
\]
for every $n \ge 1$.
It suffices therefore to show that 
$\{\Ker(\frac{\wt{\Omega}^{3}_{{Q_n}/k}}{d(\wt{\Omega}^{2}_{{Q_n}/k})}
\to H^0(E_n, \wt{\sHC}^{k,(3)}_{3,E_n}))\}_n \neq 0$.
We have shown in the proof of \thmref{thm:Main-alg-K} that
$\{H^0(E_n, \wt{\sHC}^{k,(3)}_{3,E_n}))\}_n = 0$.
Since $\wt{\Omega}^{i}_{{Q_n}/k} = {\Omega}^{i}_{{Q_n}/k}$ for every $i \ge 0$,
we are reduced to showing that 
$\{\frac{{\Omega}^{3}_{{Q_n}/k}}{d({\Omega}^{2}_{{Q_n}/k})}\}_n \neq 0$.

Since $\frac{{\Omega}^{3}_{{Q_n}/k}}{d({\Omega}^{2}_{{Q_n}/k})} \surj 
\Omega^4_{{Q_n}/k}$
for every $n \ge 1$, as shown in the proof of \thmref{thm:Main-alg-K}, 
it suffices to show that $\{\Omega^4_{{Q_n}/k}\}_n \neq 0$.

Using an analogue of ~\eqref{eqn:Alg-2**} for the surjection
$Q \surj Q_n$ and \lemref{lem:ext-pro}, we see that the map of pro-$Q$-modules 
$\{\Omega^4_{Q/k} \otimes_Q Q_n\}_n \xrightarrow{\simeq}
\{\Omega^4_{{Q_n}/k}\}_n$ is an isomorphism.
We thus need to show that
$\{\Omega^4_{Q/k} \otimes_Q Q_n\}_n \neq 0$.

It follows from \lemref{lem:No-reg-step} that for every $n \ge 1$, there is a
commutative diagram
\begin{equation}\label{eqn:No-regular-0}
\xymatrix@C1pc{
(\Omega^4_{B/k} \otimes_B Q) \otimes_Q Q_n \ar@{->>}[r] \ar@{->>}[d] & 
\Omega^4_{Q/k} \otimes_Q Q_n \ar[d] \ar[dr] & \\
(\Omega^4_{B/k} \otimes_B {B}/{\fm}) \otimes_Q Q_n \ar[ur]^-{\simeq} 
\ar[r]_-{\simeq} & \Omega^4_{B/k} \otimes_B {B}/{\fm} \ar[r]_-{\simeq} &
k \omega }
\end{equation}
such that the diagonal arrow going up in the middle is an isomorphism.
The bottom left horizontal arrow is an isomorphism
because $(f) + \fm^n \subset \fm$. We conclude that 
$\{\Omega^4_{Q/k} \otimes_Q Q_n\}_n \simeq k \omega$. This finishes the proof
of the theorem.
\end{proof}

\subsection{Lindel's question}
In order to partially answer Lindel's question on Serre dimension,
we shall repeatedly use the following elementary and folklore result.
Recall that a projective module $P$ over a Noetherian ring $A$ is
said to have a {\sl unimodular element}, if it admits a free direct summand of 
positive rank.

\begin{lem}\label{lem:red-elm}
Let $A$ be a Noetherian ring and let $P$ be a finitely generated projective 
$A$-module. Let $I \subset A$ be a nilpotent ideal such that
${P}/{IP}$ has a unimodular element. Then $P$ has a unimodular element.
\end{lem}
\begin{proof}
Our assumption is equivalent to saying that there is a surjective
$A$-linear map $\ov{\alpha}: P \surj {A}/{I}$. But such a map gives rise to a
commutative diagram
\begin{equation}\label{eqn:lift}
\xymatrix@C1pc{
P \ar[r]^{\alpha} \ar@{->>}[dr]_{\ov{\alpha}} & A \ar@{->>}[d] \\
& {A}/{I}.}
\end{equation}

Letting $M = {\rm Coker}(\alpha)$, our assumption says that
${M}/{IM} = 0$. The Nakayama lemma now implies that
there is an element $s \in I$ such that $(1+s)M = 0$.
But this implies that $M$ must be zero as $I$ is nilpotent and
hence $(1+s) \in A^{\times}$.
\end{proof}

{\sl Proof of \thmref{thm:Main-2}:}
The proof broadly follows the steps of \cite{KP}, where similar
results are proven for a more restrictive class of monoids. In view of Lemma~\ref{lem:red-elm},
we can assume that $R$ is reduced.
We first assume that $R$ has finite normalization.
Let $S$ be the normalization of $R$ and  $C$  the conductor ideal of the 
extension $R\subset S$. 
It is easy to check that height of $C\geq 1$.
Let $R'$ and $S'$ denote the quotient rings $R/C$ and $S/C$ respectively. 
Consider the two Milnor squares:
\[
\xymatrix@C1pc{
R \ar@{->}[r]          
     \ar@{->}[d]
&S 
     \ar@{->}[d]       && R[M]\ar@{->}[r]\ar@{->}[d] & S[M] \ar@{->}[d]
\\
R' \ar@{->}[r]
     &S'                && R'[M]\ar@{->}[r] & S'[M]. 
}
\]

Since $M$ is a normal monoid, $S[M]$ is the normalization of $R[M]$
with conductor ideal $C[M]$. 
Note that  $R'$ and $S'$ are zero dimensional rings.
If a projective module $P$ is free after going modulo a nilpotent ideal, 
then $P$ is free, by Nakayama lemma.
Since reduced zero dimensional rings are product of fields, 
projective modules over $R'[M],\,S'[M]$ are free by
Gubeladze \cite[Theorem~2.1]{Gub-4} (see also \cite[Theorem~1.4]{Gub-2}). 

Let $P$ be a projective $R[M]$-module of rank $r \ge 2$. 
Define $Q:=\wedge^rP\oplus R[M]^{r-1}$. Note that 
$P\otimes_{R[M]} S[M]\simeq Q\otimes_{R[M]} S[M]$ by \cite[Theorem~1.4]{Gub-2} 
and $P\otimes_{R[M]} R'[M]=R'[M]^r \simeq Q\otimes_{R[M]} R'[M]$. Let 
$\sigma_1: P\otimes_{R[M]} S[M]\rightarrow Q \otimes_{R[M]} S[M]$ and 
$\sigma_2:R'[M]^r\to R'[M]^r$
be two isomorphisms. 
Let $\phi_1:P\otimes_{R[M]} S'[M]\to S'[M]^r$ and 
$\phi_2:Q\otimes_{R[M]} S'[M]\to S'[M]^r$ be
two isomorphisms, where $\phi_1$ and $\phi_2$ are induced from $\sigma_1$ and 
$\sigma_2$. Then $P\otimes_{R[M]} S'[M]$ and $Q\otimes_{R[M]} S'[M]$ 
 are connected by a matrix $\phi:=\phi_2^{-1}\phi_1\in \GL_r(S'[M])$.
If we can prove that $\phi$ can be transformed to identity by changing the 
isomorphism at
the corner $S[M]$ or $R'[M]$, then by Milnor patching, we will have
$P\simeq Q$  and this will complete the proof.
We do it as follows.

By \cite[Lemma 4.1]{K}, we have ${\rm det}(\phi)\in U(R'[M])U(S[M])$. 
Let ${\rm det}(\phi)=u_1'u_2'$. 
Define 

\begin{center}
$u_1=\left[
\begin{array}{cccc}
u_1' & 0&\cdots & 0\\
0 & 1 & \cdots & 0\\
\vdots& \vdots&\cdots &\vdots\\
0 & 0 &\cdots & 1
\end{array}\right]$
and $u_2=\left[
\begin{array}{cccc}
u_2'^{-1} & 0&\cdots & 0\\
0 & 1 & \cdots & 0\\
\vdots& \vdots&\cdots &\vdots\\
0 & 0 &\cdots & 1
\end{array}\right].$
 \end{center}

Now changing $\sigma_1$
by $u_1^{-1}\sigma_1$ and $\sigma_2$ by $u_2^{-1}\sigma_2$, we can assume that 
$\phi\in \SL_r(S'[M])$. Since $SK_1(S'[M]) \simeq SK_1(S'[M]_{\rm red})$ 
(see, for instance, \cite[Chap.~IX, Propositions~1.3, 1.9]{Bass}), 
it follows from \thmref{thm:Main-1} that $SK_1(S'[M])=0$. 
Hence, the image of $\phi$ (say, $\wt \phi$),  will belong to $E_n(S'[M])$
(the group generated by elementary matrices) for some $n\ge r$. 
We can lift $\wt \phi$ to $E_n(S[M])$. Suitably
enlarging $\sigma_1$ and then changing it to $\sigma_1 \wt \phi^{-1}$,
the new $\phi$ (the image of $\sigma_1 \wt \phi^{-1}$) 
becomes identity. This finishes the proof when $R$ has finite normalization.

Let now $R$ be any Noetherian 1-dimensional ring and let $P$ be a
(finitely generated) projective $R[M]$-module of rank $r \ge 2$.
Lemma~\ref{lem:red-elm} says that we can assume $R$ to be reduced
in order to show that $P$ has a unimodular element.
We solve this case by reducing the problem to the case of finite normalization
using a trick of Roy \cite{Roy}. It goes as follows.

The total quotient ring $K$ of $R$ is a finite product of fields
and hence $P \otimes_{R[M]} K[M]$ is a free $K[M]$-module of rank $r$ by
\cite[Theorem~2.1]{Gub-4} (see also \cite[Theorem~1.1]{Gub-2}. 
Hence, we can find
a non-zero divisor $s \in R$ such that $P \otimes_{R[M]} R_s[M]$ is free of rank 
$r$. 

Let $\widehat{R}$
denote the $s$-adic completion $R$. Since $R \to \widehat{R}$ is flat map of 
$\Q$-algebras (see, for instance, \cite[Theorem~8.8]{Mat}), it follows that
$R[M] \to \widehat{R}[M]$ is also flat. In particular, $s \in \widehat{R}[M]$
is a non-zero divisor. 
Since the fiber squares of rings are preserved by adjoining the monoid $M$,
the analytic isomorphism
$R/sR \xrightarrow{\simeq} {\widehat{R}}/{s\widehat{R}}$ 
implies that there is a fiber square 
\begin{equation}\label{eqn:Lindel-0}
\xymatrix@C1pc{
R[M] \ar[r] \ar[d] & \widehat{R}[M] \ar[d] \\
R_s[M] \ar[r] & \widehat{R}_s[M].}
\end{equation}

Note that $\widehat{R}$ is same as the completion of $R$ along
the radical of ${(s)}$ and hence, it is the product of the completions of
$R$ along the maximal ideals $\{\fm_1, \cdots , \fm_t\}$ which are
the minimal primes of $(s)$ (see, for instance, \cite[Theorem~8.15]{Mat}).
Since $s \in {\rm rad}(\widehat{R})$ is a non-zero divisor, it follows that 
$(\widehat{R}_s)_{\rm red}$ is a product of fields. It follows from
\cite[Theorem~2.1]{Gub-4} and \lemref{lem:red-elm} that
$P \otimes_{R[M]} {\widehat{R}_s[M]}$ is a free 
$\widehat{R}_s[M]$-module of rank $r$.
We conclude that $P\otimes_{R[M]}R_s[M]\otimes_{R_s[M]} \widehat{R}_s[M]$ and 
$P\otimes_{R[M]}\widehat{R}[M] \otimes_{\widehat{R}[M]} \widehat{R}_s[M]$
are free of rank $r$ and they are canonically the 
same as $P\otimes_{R[M]} \widehat{R}_s[M]$.

It follows from \thmref{thm:Main-1} that 
$SK_1(\widehat{R}_s[M]) \simeq SK_1((\widehat{R}_s)_{\rm red}[M]) = 0$.
In particular, $E_n(\widehat{R}_s[M]) = \SL_n(\widehat{R}_s[M])$
for all $n \gg r$.
Let $P' = P \oplus {R[M]}^{n-r}$ and $Q = \wedge^n(P') \oplus {R[M]}^{n-1}$.
Since $P\otimes_{R[M]}R_s[M]$ is free, it follows that 
$P' \otimes_{R[M]}R_s[M] \simeq Q \otimes_{R[M]}R_s[M]$.

Since $\widehat{R}$ is a finite product of Noetherian 1-dimensional
complete local rings, it is classically known (see \cite{Krull} or
\cite[p.~263]{Mat}) that $\widehat{R}_{\rm red}$ has finite normalization.
In particular, it follows from \lemref{lem:red-elm}, and
the above proof of the theorem in the finite normalization case, that
every projective module over $\widehat{R}[M]$ of rank at least two
has a unimodular element. It follows that $Q \otimes_{R[M]} \widehat{R}[M]$
has a unimodular element. We therefore conclude from 
\cite[Proposition~3.4]{Roy} that $P' \simeq Q$.

We can thus write $P \oplus {R[M]}^{n-r} \simeq  
(\wedge^n(P') \oplus {R[M]}^{r-1}) \oplus {R[M]}^{n-r}$.
Finally, one checks at once that Gubeladze's monoid satisfies the
conditions of the cancellation theorem \cite[Theorem~4.4]{Keshari}.
We conclude that $P \simeq \wedge^n(P') \oplus {R[M]}^{r-1}$.
But the latter module is isomorphic to $\wedge^r(P) \oplus {R[M]}^{r-1}$. 
This finishes the proof of the theorem.
$\hfill\square$

\vskip .3cm

\noindent\emph{Acknowledgments.}
The authors would like to thank the referee for carefully reading the paper
and providing valuable suggestions to improve 
its presentation.

\end{document}